\PassOptionsToPackage{lowtilde}{url}		
\documentclass[11pt, reqno]{amsart}  
\usepackage{fancyhdr}
\usepackage{amsfonts}
\usepackage{amsmath} 
\usepackage{amssymb}
\usepackage{comment}
\usepackage{float}
\usepackage[dvipsnames]{xcolor}

\usepackage{orcidlink}

\definecolor{ikkonzome}{rgb}	{	0.9961	,	0.7569	,	0.7373	}
\definecolor{ishitake}{rgb}	{	0.9961	,	0.6941	,	0.7059	}
\definecolor{momo}{rgb}	        {	0.9961	,	0.6824	,	0.8039	}
\definecolor{kobai}{rgb}	{	0.9412	,	0.4235	,	0.5569	}
\definecolor{nakabeni}{rgb}	{	0.9451	,	0.2706	,	0.4941	}
\definecolor{sakura}{rgb}	{	0.9333	,	0.8353	,	0.8353	}
\definecolor{arazome}{rgb}	{	0.9725	,	0.7216	,	0.7843	}
\definecolor{usubeni}{rgb}	{	0.8471	,	0.4471	,	0.5451	}
\definecolor{hisame}{rgb}	{	0.7451	,	0.4039	,	0.4039	}
\definecolor{toki}{rgb}	        {	0.9569	,	0.6431	,	0.6353	}
\definecolor{sakuranezumi}{rgb}	{	0.6941	,	0.6039	,	0.6078	}
\definecolor{sango}	{rgb}	{	0.8471	,	0.4157	,	0.3725	}
\definecolor{akane}	{rgb}	{	0.7529	,	0.0118	,	0.3451	}
\definecolor{choshun}{rgb}	{	0.7490	,	0.5255	,	0.5255	}
\definecolor{karakurenai}{rgb}	{	0.7373	,	0.0118	,	0.2667	}
\definecolor{enji}{rgb}	        {	0.6275	,	0.0863	,	0.3176	}
\definecolor{keshiaka}{rgb}	{	0.6275	,	0.4275	,	0.4275	}
\definecolor{kokiake}{rgb}	{	0.5059	,	0.0706	,	0.2549	}
\definecolor{jinzamomi}{rgb}	{	0.9098	,	0.4549	,	0.4157	}
\definecolor{mizugaki}{rgb}	{	0.7294	,	0.5529	,	0.4784	}
\definecolor{umenezumi}{rgb}	{	0.5882	,	0.3922	,	0.3882	}
\definecolor{suoko}{rgb}        {	0.5843	,	0.2667	,	0.2431	}
\definecolor{akabeni}{rgb}	{	0.8039	,	0.0784	,	0.3725	}
\definecolor{shinshu}{rgb}	{	0.6431	,	0.0353	,	0.0000	}
\definecolor{azuki}{rgb}	{	0.5255	,	0.0235	,	0.0000	}
\definecolor{ginshu}{rgb}	{	0.7490	,	0.2824	,	0.0588	}
\definecolor{ebicha}{rgb}	{	0.4549	,	0.2706	,	0.2627	}
\definecolor{kuriume}{rgb}	{	0.5843	,	0.3804	,	0.4314	}
\definecolor{akebono}{rgb}	{	0.8902	,	0.5961	,	0.4941	}
\definecolor{hanezu}{rgb}	{	0.7882	,	0.5961	,	0.5373	}
\definecolor{sangoshu}{rgb}	{	0.8196	,	0.5059	,	0.4471	}
\definecolor{shozyohi}{rgb}	{	0.7686	,	0.0000	,	0.0000	}
\definecolor{shikancha}{rgb}	{	0.5569	,	0.3294	,	0.1882	}
\definecolor{kakishibu}{rgb}	{	0.6745	,	0.4078	,	0.3333	}
\definecolor{benikaba}{rgb}	{	0.7137	,	0.3373	,	0.2941	}
\definecolor{benitobi}{rgb}	{	0.6196	,	0.3176	,	0.2706	}
\definecolor{benihihada}{rgb}	{	0.5020	,	0.3137	,	0.2353	}
\definecolor{kurotobi}{rgb}	{	0.3176	,	0.2000	,	0.1490	}
\definecolor{benihi}{rgb}	{	0.8235	,	0.4745	,	0.1922	}
\definecolor{terigaki}{rgb}	{	0.8118	,	0.4627	,	0.1804	}
\definecolor{ake}{rgb}	        {	0.7804	,	0.3098	,	0.1725	}
\definecolor{edocha}{rgb}	{	0.6863	,	0.4353	,	0.2941	}
\definecolor{bengara}{rgb}	{	0.6392	,	0.1569	,	0.0196	}
\definecolor{hihada}{rgb}	{	0.5412	,	0.3412	,	0.2353	}
\definecolor{shishi}{rgb}	{	0.8549	,	0.6863	,	0.5961	}
\definecolor{araishu}{rgb}	{	0.9294	,	0.4902	,	0.4549	}
\definecolor{akago}{rgb}	{	0.8118	,	0.5765	,	0.4275	}
\definecolor{tokigaracha}{rgb}	{	0.7922	,	0.5255	,	0.3686	}
\definecolor{otan}{rgb}	        {	0.8157	,	0.4157	,	0.2235	}
\definecolor{komugi}{rgb}	{	0.8157	,	0.6549	,	0.5098	}
\definecolor{rakuda}{rgb}	{	0.6784	,	0.5255	,	0.4118	}
\definecolor{tsurubami}{rgb}	{	0.6275	,	0.4392	,	0.3961	}
\definecolor{ama}{rgb}	        {	0.7765	,	0.6902	,	0.5843	}
\definecolor{nikkei}{rgb}	{	0.7216	,	0.4667	,	0.3725	}
\definecolor{renga}{rgb}	{	0.6902	,	0.3765	,	0.3098	}
\definecolor{sohi}{rgb}   	{	0.8078	,	0.5098	,	0.2078	}
\definecolor{enshucha}{rgb}	{	0.6706	,	0.4275	,	0.1608	}
\definecolor{karacha}{rgb}	{	0.5765	,	0.4235	,	0.1490	}
\definecolor{kabacha}{rgb}	{	0.6353	,	0.3725	,	0.1569	}
\definecolor{sodenkaracha}{rgb}	{	0.5216	,	0.3490	,	0.1373	}
\definecolor{suzumecha}{rgb}	{	0.4745	,	0.3176	,	0.1255	}
\definecolor{kurikawacha}{rgb}	{	0.4078	,	0.2745	,	0.1098	}
\definecolor{momoshiocha}{rgb}	{	0.3490	,	0.2353	,	0.0902	}
\definecolor{tobi}{rgb}	        {	0.4353	,	0.3098	,	0.1412	}
\definecolor{kurumizome}{rgb}	{	0.6667	,	0.5333	,	0.3333	}
\definecolor{kaba}{rgb}	        {	0.8196	,	0.4588	,	0.1294	}
\definecolor{korosen}{rgb}	{	0.5059	,	0.3843	,	0.1608	}
\definecolor{kogecha}{rgb}	{	0.3333	,	0.2549	,	0.1098	}
\definecolor{kokikuchinashi}{rgb}	{	0.8078	,	0.5922	,	0.3490	}
\definecolor{araigaki}{rgb}	{	0.8157	,	0.5529	,	0.3176	}
\definecolor{taisha}{rgb}	{	0.6353	,	0.4196	,	0.2078	}
\definecolor{akashirotsurubami}{rgb}	{	0.8078	,	0.6118	,	0.4157	}
\definecolor{tonocha}{rgb}	{	0.5922	,	0.4039	,	0.2000	}
\definecolor{sencha}{rgb}	{	0.5255	,	0.3529	,	0.1490	}
\definecolor{sharegaki}{rgb}	{	0.8549	,	0.7098	,	0.4863	}
\definecolor{ko}{rgb}	        {	0.9529	,	0.8275	,	0.6627	}
\definecolor{usugaki}{rgb}	{	0.8627	,	0.7294	,	0.5294	}
\definecolor{koji}{rgb}	        {	0.8275	,	0.6039	,	0.2706	}
\definecolor{umezome}{rgb}	{	0.8667	,	0.7373	,	0.4235	}
\definecolor{beniukon}{rgb}	{	0.8549	,	0.5843	,	0.2549	}
\definecolor{chojicha}{rgb}	{	0.5569	,	0.4000	,	0.1098	}
\definecolor{kenpozome}{rgb}	{	0.3059	,	0.2667	,	0.0627	}
\definecolor{biwacha}{rgb}	{	0.7333	,	0.5294	,	0.1922	}
\definecolor{kohaku}{rgb}	{	0.8039	,	0.5961	,	0.2471	}
\definecolor{usuko}{rgb}	{	0.8745	,	0.7373	,	0.4824	}
\definecolor{kuchiba}{rgb}	{	0.8157	,	0.6157	,	0.3216	}
\definecolor{kincha}{rgb}	{	0.7725	,	0.4902	,	0.2000	}
\definecolor{chozizome}{rgb}	{	0.5686	,	0.3961	,	0.1608	}
\definecolor{kitsune}{rgb}	{	0.6392	,	0.4431	,	0.1804	}
\definecolor{hushizome}{rgb}	{	0.5804	,	0.4353	,	0.1608	}
\definecolor{kyara}{rgb}	{	0.4510	,	0.3373	,	0.1255	}
\definecolor{susutake}{rgb}	{	0.4588	,	0.3529	,	0.1176	}
\definecolor{shirocha}{rgb}	{	0.7529	,	0.6588	,	0.4118	}
\definecolor{odo}{rgb}	        {	0.7137	,	0.6039	,	0.3137	}
\definecolor{ginsusutake}{rgb}	{	0.5608	,	0.4745	,	0.2392	}
\definecolor{kigaracha}{rgb}	{	0.7490	,	0.6196	,	0.2745	}
\definecolor{kobicha}	{rgb}	{	0.5020	,	0.4118	,	0.1765	}
\definecolor{usuki}	{rgb}	{	0.8549	,	0.7804	,	0.4275	}
\definecolor{yamabuki}	{rgb}	{	0.9098	,	0.8000	,	0.2039	}
\definecolor{tamago}	{rgb}	{	0.8275	,	0.7412	,	0.3373	}
\definecolor{hajizome}	{rgb}	{	0.7412	,	0.6039	,	0.2431	}
\definecolor{yamabukicha}{rgb}	{	0.7059	,	0.5765	,	0.2275	}
\definecolor{kuwazome}	{rgb}	{	0.6471	,	0.5255	,	0.2118	}
\definecolor{namakabe}	{rgb}	{	0.5569	,	0.4549	,	0.1843	}
\definecolor{kuchinashi}{rgb}	{	0.8078	,	0.7333	,	0.3843	}
\definecolor{tomorokoshi}{rgb}	{	0.7804	,	0.7020	,	0.2941	}
\definecolor{shirotsurubami}{rgb}	{	0.8745	,	0.8235	,	0.5922	}
\definecolor{kitsurubami}{rgb}	{	0.6863	,	0.6039	,	0.2118	}
\definecolor{toou}	{rgb}	{	0.8353	,	0.7725	,	0.4235	}
\definecolor{hanaba}	{rgb}	{	0.8549	,	0.8000	,	0.4941	}
\definecolor{torinoko}	{rgb}	{	0.8431	,	0.8118	,	0.6902	}
\definecolor{ukon}	{rgb}	{	0.8039	,	0.7216	,	0.3059	}
\definecolor{kikuchiba}	{rgb}	{	0.7765	,	0.6667	,	0.2902	}
\definecolor{rikyushiracha}{rgb}{	0.6627	,	0.6471	,	0.4784	}
\definecolor{rikyucha}	{rgb}	{	0.5176	,	0.5176	,	0.2588	}
\definecolor{aku}	{rgb}	{	0.5294	,	0.5059	,	0.4039	}
\definecolor{higosusutake}{rgb}	{	0.4863	,	0.4510	,	0.3059	}
\definecolor{rokocha}	{rgb}	{	0.4157	,	0.4353	,	0.3686	}
\definecolor{mirucha}	{rgb}	{	0.4471	,	0.4471	,	0.3725	}
\definecolor{natane}	{rgb}	{	0.7216	,	0.6863	,	0.3882	}
\definecolor{kimirucha}	{rgb}	{	0.5373	,	0.5059	,	0.2549	}
\definecolor{uguisucha}	{rgb}	{	0.4118	,	0.3882	,	0.1961	}
\definecolor{nanohana}	{rgb}	{	0.9882	,	0.9882	,	0.3804	}
\definecolor{kariyasu}	{rgb}	{	0.8039	,	0.7686	,	0.3882	}
\definecolor{kihada}	{rgb}	{	0.9608	,	0.9137	,	0.2863	}
\definecolor{zoge}	{rgb}	{	0.8863	,	0.8235	,	0.7216	}
\definecolor{wara}	{rgb}	{	0.7882	,	0.7490	,	0.4706	}
\definecolor{macha}	{rgb}	{	0.6235	,	0.6392	,	0.4863	}
\definecolor{yamabato}	{rgb}	{	0.5137	,	0.5176	,	0.4000	}
\definecolor{mushikuri}	{rgb}	{	0.8196	,	0.8000	,	0.6118	}
\definecolor{aokuchiba}	{rgb}	{	0.6275	,	0.6392	,	0.3451	}
\definecolor{hiwacha}	{rgb}	{	0.6784	,	0.6863	,	0.4196	}
\definecolor{ominaeshi}	{rgb}	{	0.8745	,	0.8863	,	0.4039	}
\definecolor{wasabi}	{rgb}	{	0.5922	,	0.6824	,	0.5765	}
\definecolor{uguisu}	{rgb}	{	0.3333	,	0.4118	,	0.0627	}
\definecolor{hiwa}	{rgb}	{	0.7098	,	0.7216	,	0.2392	}
\definecolor{aoshirotsurubami}{rgb}	{	0.5961	,	0.6471	,	0.4471	}
\definecolor{yanagicha}	{rgb}	{	0.5373	,	0.5725	,	0.3529	}
\definecolor{rikancha}	{rgb}	{	0.3333	,	0.4196	,	0.2863	}
\definecolor{aikobicha}	{rgb}	{	0.2706	,	0.3412	,	0.2353	}
\definecolor{koke}	{rgb}	{	0.4941	,	0.5686	,	0.3922	}
\definecolor{miru}	{rgb}	{	0.1529	,	0.3216	,	0.1843	}
\definecolor{sensai}	{rgb}	{	0.1373	,	0.2863	,	0.1608	}
\definecolor{baiko}	{rgb}	{	0.5529	,	0.6157	,	0.3922	}
\definecolor{iwai}	{rgb}	{	0.3059	,	0.4118	,	0.2784	}
\definecolor{hiwamoegi}	{rgb}	{	0.4431	,	0.6824	,	0.2431	}
\definecolor{yanagisusutake}{rgb}	{	0.2039	,	0.3333	,	0.1255	}
\definecolor{urayanagi}	{rgb}	{	0.5843	,	0.6824	,	0.2431	}
\definecolor{usumoegi}	{rgb}	{	0.4824	,	0.6706	,	0.2275	}
\definecolor{yanagizome}{rgb}	{	0.4353	,	0.6000	,	0.2078	}
\definecolor{moegi}	{rgb}	{	0.3020	,	0.5961	,	0.1882	}
\definecolor{aoni}	{rgb}	{	0.1216	,	0.4196	,	0.2431	}
\definecolor{matsuba}	{rgb}	{	0.1098	,	0.3686	,	0.2118	}
\definecolor{usuao}	{rgb}	{	0.5529	,	0.7059	,	0.6118	}
\definecolor{wakatake}	{rgb}	{	0.3843	,	0.6824	,	0.5059	}
\definecolor{yanaginezumi}{rgb}	{	0.5020	,	0.6078	,	0.5490	}
\definecolor{oitake}	{rgb}	{	0.4157	,	0.5412	,	0.4392	}
\definecolor{sensaimidori}{rgb}	{	0.2549	,	0.3882	,	0.2627	}
\definecolor{midori}	{rgb}	{	0.0000	,	0.4824	,	0.0000	}
\definecolor{byakuroku}	{rgb}	{	0.6078	,	0.7333	,	0.6196	}
\definecolor{sabiseiji}	{rgb}	{	0.5333	,	0.6588	,	0.5804	}
\definecolor{rokusho}	{rgb}	{	0.3373	,	0.6039	,	0.4039	}
\definecolor{tokusa}	{rgb}	{	0.2549	,	0.4706	,	0.3922	}
\definecolor{onandocha}	{rgb}	{	0.2039	,	0.3725	,	0.3098	}
\definecolor{aotake}	{rgb}	{	0.1412	,	0.5176	,	0.4353	}
\definecolor{rikyunezumi}{rgb}	{	0.4157	,	0.5686	,	0.5490	}
\definecolor{birodo}	{rgb}	{	0.1059	,	0.4275	,	0.3608	}
\definecolor{mishiao}	{rgb}	{	0.1098	,	0.4588	,	0.3882	}
\definecolor{aimirucha}	{rgb}	{	0.2510	,	0.4000	,	0.3608	}
\definecolor{tonotya}	{rgb}	{	0.3059	,	0.4941	,	0.4392	}
\definecolor{mizuasagi}	{rgb}	{	0.1569	,	0.6745	,	0.6745	}
\definecolor{seji}	{rgb}	{	0.4745	,	0.6588	,	0.5922	}
\definecolor{seheki}	{rgb}	{	0.0588	,	0.5569	,	0.4824	}
\definecolor{sabitetsu}	{rgb}	{	0.0392	,	0.3294	,	0.2784	}
\definecolor{tetsu}	{rgb}	{	0.0627	,	0.3961	,	0.3961	}
\definecolor{omeshicha}	{rgb}	{	0.0667	,	0.4667	,	0.4667	}
\definecolor{korainando}{rgb}	{	0.0588	,	0.3922	,	0.0392	}
\definecolor{minatonezumi}{rgb}	{	0.4549	,	0.6000	,	0.6118	}
\definecolor{aonibi}	{rgb}	{	0.1804	,	0.3725	,	0.3882	}
\definecolor{tetsuonando}{rgb}	{	0.1961	,	0.4118	,	0.4275	}
\definecolor{mizu}	{rgb}	{	0.5451	,	0.7412	,	0.7608	}
\definecolor{sabiasagi}	{rgb}	{	0.4510	,	0.6000	,	0.6000	}
\definecolor{kamenozoki}{rgb}	{	0.7098	,	0.9020	,	0.8902	}
\definecolor{asagi}	{rgb}	{	0.3020	,	0.6784	,	0.7098	}
\definecolor{shinbashi}	{rgb}	{	0.0000	,	0.6000	,	0.6353	}
\definecolor{sabionando}{rgb}	{	0.2392	,	0.3961	,	0.3843	}
\definecolor{ainezumi}	{rgb}	{	0.2235	,	0.4000	,	0.3843	}
\definecolor{ai}	{rgb}	{	0.2039	,	0.3765	,	0.4314	}
\definecolor{onando}	{rgb}	{	0.1843	,	0.3686	,	0.4000	}
\definecolor{hanaasagi}	{rgb}	{	0.2000	,	0.6196	,	0.7098	}
\definecolor{chigusa}	{rgb}	{	0.1922	,	0.5725	,	0.6745	}
\definecolor{masuhana}	{rgb}	{	0.1569	,	0.4667	,	0.5569	}
\definecolor{hanada}	{rgb}	{	0.0039	,	0.4275	,	0.5373	}
\definecolor{noshimehana}{rgb}	{	0.1412	,	0.4235	,	0.5098	}
\definecolor{omeshionando}{rgb}	{	0.1176	,	0.3529	,	0.4235	}
\definecolor{sora}	{rgb}	{	0.1451	,	0.7216	,	0.8039	}
\definecolor{konpeki}	{rgb}	{	0.0902	,	0.5098	,	0.7333	}
\definecolor{kurotsurubami}{rgb}{	0.0627	,	0.3137	,	0.3451	}
\definecolor{gunjo}	{rgb}	{	0.5098	,	0.7882	,	0.9098	}
\definecolor{kon}	{rgb}	{	0.0000	,	0.2000	,	0.4000	}
\definecolor{kachi}	{rgb}	{	0.0000	,	0.1765	,	0.3490	}
\definecolor{ruri}	{rgb}	{	0.0078	,	0.3922	,	0.6510	}
\definecolor{konjo}	{rgb}	{	0.0039	,	0.3137	,	0.5216	}
\definecolor{rurikon}	{rgb}	{	0.0000	,	0.3020	,	0.5020	}
\definecolor{benimidori}{rgb}	{	0.3373	,	0.4588	,	0.6784	}
\definecolor{konkikyo}	{rgb}	{	0.0039	,	0.0667	,	0.4275	}
\definecolor{hujinezumi}{rgb}	{	0.3216	,	0.3686	,	0.6118	}
\definecolor{benikakehana}{rgb}	{	0.2275	,	0.2471	,	0.5882	}
\definecolor{hujiiro}	{rgb}	{	0.4706	,	0.4863	,	0.7059	}
\definecolor{hutaai}	{rgb}	{	0.3137	,	0.3333	,	0.5608	}
\definecolor{hujimurasaki}{rgb}	{	0.4157	,	0.3333	,	0.5686	}
\definecolor{kikyo}	{rgb}	{	0.3529	,	0.3216	,	0.5725	}
\definecolor{shion}	{rgb}	{	0.5137	,	0.4196	,	0.6784	}
\definecolor{messhi}	{rgb}	{	0.2196	,	0.1765	,	0.3098	}
\definecolor{shikon}	{rgb}	{	0.2510	,	0.1961	,	0.3529	}
\definecolor{kokimurasaki}{rgb}	{	0.3098	,	0.2000	,	0.3765	}
\definecolor{usu}	{rgb}	{	0.7294	,	0.6353	,	0.7843	}
\definecolor{hashita}	{rgb}	{	0.6000	,	0.4549	,	0.6706	}
\definecolor{rindo}	{rgb}	{	0.5922	,	0.5137	,	0.7529	}
\definecolor{sumire}	{rgb}	{	0.3882	,	0.2157	,	0.5922	}
\definecolor{nasukon}	{rgb}	{	0.3216	,	0.2235	,	0.4353	}
\definecolor{murasaki}	{rgb}	{	0.3137	,	0.1765	,	0.4824	}
\definecolor{kurobeni}	{rgb}	{	0.2353	,	0.1294	,	0.3059	}
\definecolor{ayame}	{rgb}	{	0.4275	,	0.1569	,	0.5098	}
\definecolor{benihuji}	{rgb}	{	0.6824	,	0.5333	,	0.6667	}
\definecolor{edomurasaki}{rgb}	{	0.3961	,	0.1529	,	0.4392	}
\definecolor{kodaimurasaki}{rgb}{	0.4353	,	0.2863	,	0.4627	}
\definecolor{shikon}{rgb}       {	0.4118	,	0.2549	,	0.3765	}
\definecolor{hatobanezumi}{rgb}	{	0.4235	,	0.3804	,	0.4392	}
\definecolor{budonezumi}{rgb}	{	0.3412	,	0.2196	,	0.3529	}
\definecolor{ebizome}	{rgb}	{	0.4000	,	0.2235	,	0.3882	}
\definecolor{hujisusutake}{rgb}	{	0.2549	,	0.0824	,	0.2471	}
\definecolor{usuebi}	{rgb}	{	0.6549	,	0.4392	,	0.6039	}
\definecolor{botan}	{rgb}	{	0.6706	,	0.0392	,	0.4353	}
\definecolor{umemurasaki}{rgb}	{	0.6667	,	0.3922	,	0.5176	}
\definecolor{nisemurasaki}{rgb}	{	0.3686	,	0.0000	,	0.3216	}
\definecolor{murasakitobi}{rgb}	{	0.3922	,	0.3255	,	0.3529	}
\definecolor{ususuo}{rgb}	{	0.7569	,	0.4000	,	0.5412	}
\definecolor{suo}{rgb}	        {	0.6941	,	0.0627	,	0.4157	}
\definecolor{kuwanomi}{rgb}	{	0.4196	,	0.0549	,	0.2745	}
\definecolor{nibi}{rgb}	        {	0.4549	,	0.4235	,	0.3961	}
\definecolor{benikeshi}{rgb}	{	0.4118	,	0.3804	,	0.3843	}
\definecolor{shironeri}{rgb}	{	0.9843	,	0.9843	,	0.9843	}
\definecolor{shironezumi}{rgb}	{	0.6471	,	0.6588	,	0.6627	}
\definecolor{ginnezumi}{rgb}	{	0.5255	,	0.5490	,	0.5922	}
\definecolor{sunezumi}{rgb}	{	0.4275	,	0.4392	,	0.4392	}
\definecolor{dobunezumi}{rgb}	{	0.2863	,	0.2941	,	0.2941	}
\definecolor{aisumicha}{rgb}	{	0.2078	,	0.2196	,	0.2510	}
\definecolor{binrojizome}{rgb}	{	0.2118	,	0.0824	,	0.0706	}
\definecolor{sumizome}{rgb}	{	0.2706	,	0.2706	,	0.2706	}

\usepackage{amsthm}
\usepackage{abraces}
\usepackage{mathtools}                          
\usepackage{mathrsfs}				    
\usepackage{verbatim}				    
\usepackage[all]{xy}    	          
\SelectTips{cm}{11}                  
\usepackage[T1]{fontenc}        
\usepackage[normalem]{ulem}          
\usepackage{microtype}               
\usepackage{tikz-cd}
\usepackage{graphicx, subcaption}	

\usepackage{enumitem}
\setenumerate[1]{label=(\roman*),ref=\thesubsection.{\roman*}}    

\usepackage[margin=1in]{geometry}

\usepackage{hyperref}
\hypersetup{
  colorlinks   = true,          
  urlcolor     = blue,          
  linkcolor    = RoyalBlue,          
  citecolor   = magenta             
}

\usepackage{pdfsync}
\usepackage{dsfont}
\usepackage{prettyref}

\newcommand{\pref}{\prettyref}

\newrefformat{th}{Theorem~\ref{#1}}
\newrefformat{cr}{Corollary~\ref{#1}}
\newrefformat{lm}{Lemma~\ref{#1}}
\newrefformat{dl}{Definition-Lemma~\ref{#1}}
\newrefformat{pd}{Proposition-Definition~\ref{#1}}
\newrefformat{td}{Theorem-Definition~\ref{#1}}
\newrefformat{df}{Definition~\ref{#1}}
\newrefformat{ft}{Fact~\ref{#1}}
\newrefformat{cl}{Claim~\ref{#1}}
\newrefformat{sl}{Sublemma~\ref{#1}}
\newrefformat{pr}{Proposition~\ref{#1}}
\newrefformat{cj}{Conjecture~\ref{#1}}
\newrefformat{st}{Step~\ref{#1}}
\newrefformat{sc}{\S~\ref{#1}}
\newrefformat{ss}{\S~\ref{#1}}
\newrefformat{df}{Definition~\ref{#1}}
\newrefformat{rm}{Remark~\ref{#1}}
\newrefformat{q}{Question~\ref{#1}}
\newrefformat{pb}{Problem~\ref{#1}}
\newrefformat{cd}{Condition~\ref{#1}}
\newrefformat{eg}{Example~\ref{#1}}
\newrefformat{he}{Heore~\ref{#1}}
\newrefformat{fg}{Figure~\ref{#1}}
\newrefformat{tb}{Table~\ref{#1}}
\newrefformat{as}{Assumption~\ref{#1}}
\newrefformat{ap}{Appendix~\ref{#1}}
\newrefformat{nt}{Notation~\ref{#1}}
\newcommand{\Z}{{\mathbb{Z}}}
\newcommand{\Q}{{\mathbb{Q}}}
\newcommand{\R}{\mathbb{R}}			
\newcommand{\C}{\mathbb{C}}

\newcommand{\PP}{{\mathbb P}}

\newcommand{\calG}{{\mathcal G}}
\newcommand{\calH}{{\mathcal H}}

\newcommand{\calL}{{\mathcal L}}
\newcommand{\calM}{{\mathcal M}}
\newcommand{\calN}{{\mathcal N}}
\newcommand{\calO}{{\mathcal O}}

\newcommand{\calX}{{\mathcal X}}

\newcommand{\frakG}{{\mathfrak G}}

\newcommand{\scrA}{\mathscr A}
\newcommand{\scrB}{\mathscr B}		
\newcommand{\scrM}{\mathscr M}

\newcommand{\scrL}{\mathscr L}

\newcommand{\scrH}{\mathscr H}

\newcommand{\scrX}{\mathscr X}		
\newcommand{\scrY}{\mathscr Y}

\newcommand{\CK}{{\mathbb{C}_K}} 

\newcommand{\id}{\mathrm{id}}
\newcommand{\FS}{\mathrm{FS}}
\newcommand{\diag}{\mathrm{diag}}
\newcommand{\str}{\mathrm{str}}
\newcommand{\alg}{\mathrm{alg}}
\newcommand{\surj}{\twoheadrightarrow}
\newcommand{\inj}{\hookrightarrow}

\newcommand{\na}{\mathrm{na}}
\newcommand{\hyb}{\mathrm{hyb}}
\newcommand{\trop}{\mathrm{trop}}

\newcommand{\Img}{\mathrm{Im}}

\newcommand\wc{{\mkern 2mu\cdot\mkern 2mu}}
\newcommand\va{{|\wc|}}
\newcommand\nm{{\|\wc\|}}

\newcommand{\CPSH}{\mathrm{CPSH}}

\DeclareMathOperator{\Spec}{Spec}

\DeclareMathOperator{\Hom}{Hom}
\DeclareMathOperator{\SL}{SL}
\DeclareMathOperator{\GL}{GL}
\DeclareMathOperator{\supp}{supp}

\DeclareMathOperator{\coker}{coker}
\DeclareMathOperator{\rk}{rk}

\DeclareMathOperator{\Aut}{Aut}

\DeclareMathOperator{\Sk}{Sk}
\DeclareMathOperator{\PGL}{PGL}

\DeclareMathOperator{\codim}{codim}

\newtheoremstyle{plain2}    
  {}            
  {}            
  {\itshape}    
  {}            
  {\bfseries}   
  {.}           
  {5pt plus 1pt minus 1pt}  
  {{\thmnumber{(#2)} \thmname{#1}{\thmnote{ (#3)}}}}          

\newtheorem{Thm}{Theorem}[subsection]
\newtheorem*{Thm*}{Theorem}
\newtheorem*{Cor*}{Corollary}
\newtheorem{Cor}[Thm]{Corollary}
\newtheorem{Lem}[Thm]{Lemma}
\newtheorem{Prop}[Thm]{Proposition}

\newtheorem{Theorem}{Theorem}

\theoremstyle{definition}
\newtheorem*{Def*}{Definition}

\newtheoremstyle{definition2}    
  {}   
  {}   
  {\normalfont}  
  {}       
  {\bfseries} 
  {.}        
  {5pt plus 1pt minus 1pt} 
  {{(\thmnumber{#2}) \thmname{#1}{\thmnote{#3}}}}          

\theoremstyle{definition}
\newtheorem{Def}[Thm]{Definition}

\newtheorem{example}[Thm]{Example}

\newtheorem{Q.}[Thm]{Question}
\theoremstyle{remark}
\newtheorem{Rem}[Thm]{Remark}
\newtheorem*{Rem*}{Remark}

\newtheoremstyle{stepstyle}
  {}     {}   
  {\normalfont}  
  {\parindent}       
  {\itshape} 
  {}         
  {5pt plus 1pt minus 1pt} 
  {{\thmname{#1} \thmnumber{#2}:{\thmnote{#3}}}}          
\theoremstyle{stepstyle}

\newtheoremstyle{point}
  {}     {}   
  {\normalfont}  
  {}       
  {\bfseries} 
  {}         
  {5pt plus 1pt minus 1pt} 
  {{\thmname{#1}\thmnumber{#2}.\thmnote{ #3.}}}          
\theoremstyle{point}
\newtheorem{point}[subsection]{}

\newtheoremstyle{point*}
  {}     {}   
  {\normalfont}  
  {}       
  {\bfseries} 
  {}         
  {5pt plus 1pt minus 1pt} 
  {{\thmname{#1}\thmnote{ #3.}}}          
\theoremstyle{point*}
\newtheorem{point*}[subsubsection]{}

\numberwithin{equation}{subsection}

\newtheoremstyle{subpoint}
  {}     {}            
  {\normalfont}  
  {}                   
  {\normalfont} 
  {}         
  {5pt plus 1pt minus 1pt} 
  {{\thmname{#1}(\thmnumber{#2})\thmnote{ #3.}}}          
\theoremstyle{subpoint}
\newtheorem{subpoint}[equation]{}

\usepackage{etoolbox}
\makeatletter
\patchcmd{\@maketitle}
  {\ifx\@empty\@dedicatory}
  {\ifx\@empty\@date \else {\vskip3ex \centering\footnotesize\@date\par\vskip1ex}\fi
   \ifx\@empty\@dedicatory}
  {}{}
\patchcmd{\@adminfootnotes}
  {\ifx\@empty\@date\else \@footnotetext{\@setdate}\fi}
  {}{}{}
\makeatother

\title[NA balanced metrics and their application to totally degenerate abelian varieties]{
Non-Archimedean balanced metrics and their application to totally degenerate abelian varieties
}

\author{\ \ Keita Goto${}^\mathrm{\orcidlink{0000-0002-6656-5253}}$}
\address[Keita~Goto]{Division of Fundamental Mathematical Science, 
RIKEN Center for Interdisciplinary Theoretical and Mathematical Sciences (iTHEMS),
2-1 Hirosawa, Wako, Saitama 351-0198,
JAPAN
}
\email{\href{k.goto.math@gmail.com}{k.goto.math@gmail.com}}
\urladdr{\url{https://sites.google.com/view/keitagoto}}
\keywords{non-Archimedean balanced metrics, non-Archimedean Calabi--Yau metrics, theta functions} 
\subjclass[2020]{Primary 14K25; Secondary 14G22, 32P05, 32Q26, 32U15}

\begin{document}
\begin{abstract}
For a polarized complex manifold with discrete automorphism group, it is known that if the first Chern class admits a cscK metric, then the balanced metrics, which are characterized in terms of the algebro-geometric notion of Chow stability, approximate this cscK metric.
In this paper, we study a non-Archimedean analogue of this phenomenon. In particular, we prove that such an analogue holds for polarized totally degenerate abelian varieties. As an application, we also show that, 
for a totally degenerating family of polarized abelian varieties,
the validity of this non-Archimedean analogue yields a uniform estimate for the Calabi--Yau metrics on fibers sufficiently close to the degenerate fiber.

\end{abstract}

\maketitle

\setcounter{tocdepth}{1}
\tableofcontents

\section{Introduction}

Let $X$ be a smooth Calabi--Yau variety over $K:=\C((t))$ with maximal degeneration.
Here, $K$ is a non-Archimedean field equipped with a $t$-adic valuation $\va_K$ such that $|t|_K:=e^{-1}$. 
Then we may consider a non-Archimedean analytic space $X^\na$ associated to $X$ in the sense of \cite{Ber90}.
The essential skeleton $\Sk (X) \subset X^\na$ introduced in \cite{KS06} is defined as the locus where the weight function on $X^\na$, which generalizes the log discrepancy, attains its minimum; see \cite{MN15}.
By \cite{NXY19}, if $X$ admits semistable reduction, then $\Sk (X) \subset X^\na$ carries a canonical probability measure $\mu$ induced by a non-canonical integral affine structure on $\Sk (X)\setminus \Gamma,$ where $\codim \Gamma \geq 2$.

As we will briefly recall in \pref{ss: NA pluripotential theory}, the framework of \emph{non-Archimedean pluripotential theory}, initiated by \cite{CLD12} and now well-established,
allows us to study metrics on such non-Archimedean analytic spaces; see \cite{BE21} for details.
In particular, as a non-Archimedean analogue of Calabi--Yau metrics,
for any ample line bundle $L$ on the above $X$,
we obtain a canonical metric on $L^\na$ called the
\emph{non-Archimedean Calabi--Yau metric}, abbreviated as the \emph{NACY metric} and denoted by $\phi^\mathrm{NACY}$.
More precisely, $\phi^\mathrm{NACY}$
is defined to be the solution to the  \emph{non-Archimedean Monge--Amp\`ere equation}
$$
 \mathrm{MA}(\phi)=(L^{\dim X})\mu.
$$
 Here, by \cite[Theorem~A]{BFJ15}, the NACY metric $\phi^\mathrm{NACY}$ exists uniquely (up to an additive constant).
In particular, note that $\phi^\mathrm{NACY}$ is defined analytically and is intrinsic to the pair $(X,L)$. The aim of this paper is to understand 
$\phi^\mathrm{NACY}$ from an algebraic perspective.

Recall that, for a polarized complex Calabi--Yau manifold $(X,L)$, 
 \cite[Theorem~3]{Don01} implies that 
the Calabi--Yau metric $\omega_{\phi^\mathrm{CY}}$ in the first Chern class $c_1(L)$ is approximated by balanced metrics $\omega_{\phi_k}$ at level $k$ as $k\to \infty$.
Here and throughout, we use the notation $\omega_\phi$ for the Kähler metric
associated with a Kähler potential $\phi$.
By \cite{Zha96} and \cite{Luo98}, such balanced metrics are characterized by Chow (poly)stability, a well-established notion in geometric invariant theory (GIT).
In particular, we may apply the Kempf--Ness theory initiated in \cite{KN79} to find such balanced metrics. 
More precisely,
 the balanced metric $\omega_{\phi_k}$ can be characterized by a
variational problem for Hermitian norms on $H^0(X,kL)$ with a fixed determinant, 
namely, by minimizing a functional, which we call the Chow norm, over the space of such norms.
Here, it is also worth noting that the functional appears implicitly in earlier studies of heights in Arakelov geometry; see, for instance, \cite{Sou91}, \cite{Phi91},  \cite{BGS94}, \cite{Bos94}, and \cite{Bos96}.
As demonstrated in  \cite{Bur92}, \cite{Zha96}, and \cite{Mac17},
this Kempf--Ness-type approach is particularly well suited to 
non-Archimedean pluripotential theory.
Hence, it is natural to consider a non-Archimedean analogue of balanced metrics.
In fact, Zhang had already implicitly considered such an analogue 
in
\cite{Zha96}, prior to the development of non-Archimedean pluripotential theory.
Later, within the framework of non-Archimedean pluripotential theory, Yanbo Fang
reformulated this analogue as the notion of \emph{critical metrics} in \cite{Fan22}.
However, in order to maintain the analogy with the complex case, the arguments
on critical metrics were carried out over an algebraically closed non-Archimedean
field. 
Since our field $K$ is not algebraically closed, we slightly modify this notion and
introduce \emph{non-Archimedean balanced metrics} (\emph{NA balanced metrics} for short)  as follows:

\begin{Def*}[=\pref{df: NA balanced}]
   Let $(\C_K,\va_\CK)$ be the completion of an algebraic closure $\overline{K}$ of $K$ with respect to $\va_K$, 
   and let $(X,L)$ be a
polarized $K$-variety. Fix $k\in \Z_{>0}$. A non-Archimedean Fubini--Study metric $\phi$ on $L^\na$ is
called a \emph{NA balanced metric at level $k$} if its pullback
$\phi_{\C_K}$ along the base change morphism $X_{\C_K}\to X$ is a critical
metric on $L_{\C_K}^\na$ at level $k$ in the sense of \cite{Fan22} and the metric $\phi_{\C_K}$ arises from
a model $(\scrX,\scrL)$ of $(X_{\overline{K}},kL_{\overline{K}})$ over 
$\overline{R}:=\{a\in \overline{K} \subset \CK \ |\ |a|_\CK\leq 1\}\subset \overline{K}$
such that $\scrL$ is relatively very ample on $\scrX$.
 (See
\pref{eg: algebraic NAFS and model} for details on the latter condition.)
\end{Def*}

Then, in analogy with \cite{Don01}, for the polarized Calabi--Yau variety
$(X,L)$ over $K$ considered at the beginning, we expect such non-Archimedean
balanced metrics to approximate the NACY metric.
As a first step toward verifying this expectation, we study the case of
polarized totally degenerate abelian varieties, and obtain the following:

\begin{Theorem}[=\pref{th: theta functions and NA balanced metrics}]
\label{th: A}
Let $(X,L)$  be a polarized totally degenerate abelian variety over $K$.
For all sufficiently large $k$, there exists a non-Archimedean balanced metric $\phi_k$ on $L^\na$ at level $k$.
Furthermore, the non-Archimedean balanced metric $\phi_k$ is explicitly given
 by \eqref{eq: NA balanced metric for pav}.
\end{Theorem}
\begin{Rem*}
Note that our non-Archimedean balanced metric $\phi_k$ is characterized by \emph{(normalized) theta functions with characteristics} in the sense of 
\pref{df: normalized theta at level k}. As for $k$, the only requirement is that $kL$ be very ample.
 Thus it suffices to assume that $k\geq 3$.
As for the uniqueness of an NA balanced metric, \pref{th: res Chow stable implies a unique NA balanced metric} provides a purely algebro-geometric sufficient condition, which we call \emph{residual Chow stability}, for the existence and uniqueness of an NA balanced metric, just as Chow stability over $\C$ implies the existence and uniqueness of a balanced metric. However, the polarized abelian variety considered here may not satisfy this condition (cf.~\pref{eg: Aut of the special fiber may have positive dimension}). Thus, uniqueness in this case remains unclear, even though our NA balanced metrics enjoy desirable properties, as shown in
\pref{th: B} and \pref{th: C}.
\end{Rem*}

\begin{Theorem}[=\pref{th: convergence of NA balanced metrics}]
\label{th: B}
Keep the same notation as above.
 These non-Archimedean balanced metrics $\phi_k$ on $L^\na$ converge  to the NACY metric $\phi^\mathrm{NACY}$ on $L^\na$ in the $\mathscr{C}^0$-topology as $k\to \infty$.
Furthermore, we have
    $$ \sup_{x\in X^\na} \left| \phi^\mathrm{NACY}-\phi_k\right| = O(k^{-2}).$$
\end{Theorem}
\begin{Rem*}
It is well known that $\phi^\mathrm{NACY}$ is the toric metric associated with
the \emph{Legendre dual} of a positive definite quadratic form $A_\Phi$.
    As a key lemma, we show that our 
    $\phi_k$ is the toric metric associated with what we call
the \emph{discrete Legendre dual} of $A_\Phi$ at level $k$; see \pref{lm: discrete Legendre dual}.
\end{Rem*}
As an application, we consider a
totally degenerating family $(\calX,\calL)$ of polarized abelian varieties over the punctured disc $\Delta^*:=\{t\in \C \ |\ 0<|t|<e^{-1}\}$. Then, as in \cite{BJ17}, the \emph{hybrid analytic space $\calX^\hyb$} associated to $\calX$ admits a fibration
$$
    \pi^\hyb:\calX^\mathrm{hyb}\to \Delta=\Delta^*\cup \{ 0\}
$$
such that there are canonical isomorphisms $h:\calX^\mathrm{hyb}|_{\Delta^*}\simeq \calX|_{\Delta^*}$ and $\calX_0^\hyb:=(\pi^\hyb)^{-1}(0)\simeq \calX_K^\na$. In particular, for any $t\in \Delta^*$, 
an induced isomorphism $h_t:\calX^\hyb_t:=(\pi^\hyb)^{-1}(t)\simeq \calX_t$
can be regarded as a rescaling by $$c(t):=\frac{-1}{\log |t|},$$ and
 $\calL_t^\hyb$ can be regarded as $c(t)\calL_t$ on $\calX_t$. 
 In addition, we have a canonical isomorphism $\calL^\hyb_0\simeq \calL_K^\na$.
As we will briefly recall in \pref{ss: hybrid spaces}, 
the framework of non-Archimedean pluripotential theory can be extended to 
more general Berkovich analytic spaces, including hybrid analytic spaces;
see \cite{Fav20} and \cite{PS23} for details.
In particular, there is a notion of \emph{hybrid cpsh metrics} on $\calL^\hyb$, compatible both with cpsh metrics on the complex fibers $\calL=\{\calL_t\}_{t\in \Delta^*}$ and with cpsh metrics on the non-Archimedean fiber $\calL_K^\na$.
More precisely, for any $t\in \Delta$ and any hybrid cpsh metric $\phi\in \CPSH(\calL^\hyb)$, denote by $\phi(t)$ its restriction to $\calX^\hyb_t$. Then the  hybrid cpsh metric $\phi\in \CPSH(\calL^\hyb)$ satisfies
    \begin{equation*}
        \phi (t) \in 
  \begin{cases*}
   \CPSH(\calL_t^\hyb)=\CPSH(c(t)\calL_t) & if $t\neq 0$, \\
   \CPSH(\calL_0^\hyb)=\CPSH(\calL_K^\na)     & if $t=0$.
  \end{cases*} 
    \end{equation*} 
To keep the notation compatible with the scaling on the complex side,
for $t\in \Delta^*$,
we sometimes denote a cpsh metric $c(t)^{-1}\phi(t)$ on $\calL_t$ by $\phi_t$.
 Then we can apply \pref{th: A} and \pref{th: B} to  $(X,L):=(\calX_K,\calL_K)$.
Combining the explicit description of complex balanced metrics on  $\calL$ in \cite{WZ17} with the corresponding non-Archimedean description given in \pref{th: A}, we obtain the following by a direct computation:

\begin{Theorem}[=\pref{th: continuity of hybrid balanced metric}]
\label{th: C}
For all sufficiently large $k$, after possibly making a finite base change,
    there exists a hybrid cpsh metric $$\phi_k^\hyb\in \CPSH (\calL^\hyb)$$ such that,
    for each $t\in \Delta^*$, the restriction
    $\phi_{k,t}^\hyb:=c(t)^{-1}\phi_k^\hyb(t)$ is the complex balanced metric on $\calL_t$ at level $k$  and, for $t=0$, the restriction $\phi_k^\hyb(0)$ is the NA balanced metric on $\calL_K^\na$ at level $k$.
\end{Theorem}
\begin{Rem*}
    Note that the above finite base change is used only to ensure the existence of a family of balanced metrics. In particular, since our non-Archimedean balanced metric is defined by using a NAFS metric over $\CK$, it is independent of this base change.
    
\end{Rem*}

Furthermore, exploiting the topological properties of $\calX^\hyb$, \pref{th: B} implies the following uniform estimate:

\begin{Theorem}[=a consequence of \pref{th: approximation of CY by L^p-metrics}]
\label{th: D}
For all sufficiently large $k\in\Z_{>0}$, after possibly
making a finite base change and 
shrinking the radius of $\Delta$,
     we obtain the uniform estimate
    $$ \sup_{\calX_t} c(t)|\phi_t^\mathrm{CY} - \phi_{k,t}^\hyb|= O(k^{-2})$$
    for all 
    $t\in \Delta^*$, where $\phi^\mathrm{CY}_t$ is the K\"ahler potential associated with the Calabi--Yau metric on $\calL_t$.
\end{Theorem}

\begin{Rem*}
The finite base change above is required only when applying \pref{th: C}.
In particular, the proof of \pref{th: D} does not require any base change. Note, however, that the radius of the shrunken disc may depend on $k$, and it remains unclear whether the disc can be chosen independently of $k$. 

\end{Rem*}

In summary, we emphasize the following consequence for  a
totally degenerating family 
$(\calX,\calL)$
of polarized 
abelian varieties over $\Delta^*$, obtained by combining \pref{th: B} and \pref{th: C} with \cite[Theorem~3]{Don01} and \cite[Theorem~4.14]{GO24}:
$$\phi^\mathrm{CY}_t=\lim_{k\to \infty} \phi_{k,t}, \quad \phi_k=\lim_{t\to 0} c(t)\phi_{k,t}, \quad \phi^\mathrm{NACY}=\lim_{k\to \infty} \phi_{k}, \quad \phi^\mathrm{NACY}=\lim_{t\to 0} c(t)\phi_{t}^\mathrm{CY},$$
where
\begin{itemize}
    \item $\phi_{k,t}$ is the complex balanced metric on $\calL_t$ at level $k$,
    \item $\phi_t^\mathrm{CY}$ is the Calabi--Yau metric on $\calL_t$,
    \item $\phi_k$ is the NA balanced metric on $L^\na$ at level $k$, 
    \item $\phi^\mathrm{NACY}$ is the NACY metric on $L^\na$.
\end{itemize}
Note that only the equality $\phi_k=\lim_{t\to 0} c(t)\phi_{k,t}$ may require a finite base change.
This consequence highlights a characteristic property of the NA balanced metrics introduced in this paper.

\subsection*{Structure of the paper}
The paper is organized as follows:
\begin{itemize}
    \item  
    \pref{sc:Preliminaries} is devoted to preliminaries. 
    Thus, most of the material there is not mathematically new. 
    However, for later use, we introduce several nonstandard classes of 
    non-Archimedean Fubini--Study metrics, 
    such as $\FS_k(L^\na)$, $\FS_k^\str(L^\na)$ and $\FS_k^\alg(L^\na)$.

    \item \pref{sc:NA balanced on totally degenerate abelian varieties} forms the technical heart of this paper. There,
      we introduce non-Archimedean balanced metrics, 
    which are the main objects studied in this paper, and
     study polarized totally degenerate abelian varieties explicitly in terms of degeneration data in the sense of \cite{FC}. In particular, we prove \pref{th: A} and \pref{th: B}.

    \item In \pref{sc:Applications via hybrid geometry}, 
    we briefly recall some basic notions of hybrid geometry and
    consider applications of the non-Archimedean results obtained in \pref{sc:NA balanced on totally degenerate abelian varieties}. In particular, we prove \pref{th: C} and \pref{th: D} using \pref{lm: usc for sup}, a lemma based on a purely topological argument.
\end{itemize}

\subsection*{Notation and conventions}
We now collect some notation and conventions that will be used throughout the paper.
    \begin{enumerate}[label=(\roman*)]
    \item For any field $F$,  its \emph{algebraic closure} is denoted by $\overline{F}$.
     \item For any scheme $X$ over a ring $A$,
     the \emph{structure sheaf on $X$} is denoted by  $\calO_X$.
     By a \emph{line bundle $L$ on $X$}, we mean an invertible sheaf on $X$.
     Here, we use additive notation for line bundles; for instance, $kL$ denotes $L^{\otimes k}$ for $k\in \Z_{>0}$.
     For any ring $R$, the set of $R$-valued points of $X$ is denoted by $X(R)$.
     In addition, for any ring homomorphism $A\to B$, we write $X_B:=X\times_{\Spec A}\Spec B$ for the \emph{base change of $X$ to $B$}. Similarly, for any line bundle $L$ on $X$, we write $L_B$ for the \emph{base change of $L$ to $X_B$}.
      \item A \emph{polarized variety $(X,L)$} over a field $F$ is a pair 
     consisting of an $F$-variety $X$ and an ample line bundle $L$ on $X$. 
    We denote by $$(L^d)\in \Z_{>0}$$ the \emph{top self-intersection number of $L$}, where $d:=\dim X$. 
    
    \item 
     A \emph{non-Archimedean field} is a field equipped with a non-Archimedean valuation with respect to which it is complete. We use multiplicative notation for all valuations in this paper. For a non-Archimedean field $F$, we denote its valuation by $|\cdot|_F$. We denote by $F^\circ$ the \emph{valuation ring} of $F$, namely,
     $$F^\circ:=\{a\in F \ | \ |a|_F\leq 1\}.$$ 
     The \emph{residue field} of $F$ is denoted by $\widetilde{F}$, namely,
     $$\widetilde{F}:=F^\circ/F^{\circ\circ},$$ 
     where $F^{\circ\circ}:=\{a\in F \ |\ |a|_F<1\}$ is the unique maximal ideal of $F^\circ$. 
     Finally, the \emph{value group} is denoted by $|F^\times|$, namely,
     $$|F^\times|:=\{|a|_F\in \R_{>0} \ |\ a\in F^\times\}\subset \R_{>0}.$$
     
     \item For any scheme $X$ locally of finite type over a non-Archimedean field $F$, let $X^\na$ denote the \emph{$F$-analytic space associated to $X$}
      in the sense of Berkovich; see \cite{Ber90}. For any line bundle $L$ on $X$, we denote by $L^\na$ 
     the \emph{induced line bundle on $X^\na$}. For each $x\in X^\na$, the \emph{completed residue field} is denoted by $\scrH(x)$.
     For a complete field extension $F'/F$, we denote by $X_{F'}^\na$ the \emph{$F'$-analytic space associated to $X_{F'}$}. Similarly, we denote by $L_{F'}^\na$ the \emph{induced line bundle on $X_{F'}^\na$}. On the other hand, for any scheme $X$ locally of finite type over $\C$, we often identify $X$ with its associated complex analytic space $X(\C)$, and a line bundle $L$ on $X$ with its associated line bundle $L(\C)$ on $X(\C)$. Furthermore, if $X$ is smooth over $\C$, then we often identify a Kähler potential $\phi$ with its associated Kähler metric
     $\omega_\phi$ on $X(\C)$.

     \item 
     For $n\in \Z_{>0}$, let $\GL_n$ (resp. $\SL_n$) denote 
     the \emph{general linear group scheme} (resp. the \emph{special linear group scheme}) over $\Z$. 
     In particular, for any ring $A$, we denote by $\GL_n(A)$ (resp. $\SL_n(A)$) 
     the \emph{group of $A$-valued points of $\GL_n$} (resp. $\SL_n$). We also write $U(n)\subset \GL_n(\C)$ for the \emph{unitary group}.

    \end{enumerate}

\subsection*{Acknowledgements}
The idea for this paper originated when the author read a paper by Yanbo Fang  \cite{Fan22}. The author is grateful to him for fruitful discussions in Lille about the initial ideas for this paper, on the occasion of the workshop 
\emph{Non-Archimedean Geometry}, held from October 6 to 10, 2025,
and for subsequent discussions.
The author also thanks 
Paul Alexander Helminck, Itsuki Tazoe, Yuki Tsutsui, and Yuto Yamamoto
for various valuable comments on this work.
In particular,
the author would like to express sincere gratitude to Yuji Odaka both for his interest in this work and for encouraging the author to develop it into a paper. 
In addition, the author would like to thank
Kazushi Ueda for his continued support during the period when the author was affiliated with the University of Tokyo.
The author also thanks Rin Gotou for bringing their recent result \cite{GO26} to his attention during the revision of this manuscript.
This work was supported by JSPS KAKENHI Grant Number JP25KJ0084 and RIKEN iTHEMS Program.

\subsection*{Use of Generative AI}
The author used ChatGPT 
as a supplementary tool for proofreading the manuscript, reviewing relevant literature, and checking the author's arguments for potential errors.
All mathematical statements and proofs were independently verified by the author.
\section{Preliminaries}
\label{sc:Preliminaries}
\subsection{A brief review of non-Archimedean pluripotential theory}
\label{ss: NA pluripotential theory}
Berkovich geometry was introduced by Berkovich in \cite{Ber90} as a framework for the analytic study of algebraic varieties over non-Archimedean fields. It  has attracted significant attention in recent years because of its many applications. 
In particular, notions analogous to metrics and currents in Kähler geometry were introduced by Chambert-Loir and Ducros in \cite{CLD12}, and later developed into a systematic theory, now known as \emph{non-Archimedean pluripotential theory}, through the contributions of many mathematicians, especially Boucksom, Eriksson, Favre, and Jonsson.
In this subsection, we review some notions and results concerning this theory that will be needed later in the paper. A large part of this subsection is based on \cite{BE21}.

\begin{Def}[cf. {\cite[\S~2.5]{BGR}} and {\cite[Definition~1.10]{BE21}}]
Let $F:=(F,\va_F)$ be a non-Archimedean field, and let $V$ be an 
    $F$-vector space.
    An \emph{$F$-norm $\nm$ on $V$} is a function $\nm: V\to \R_{\geq 0}$ such that
    \begin{itemize}
        \item $||av||=|a|_F||v||$ for all $a\in F, v\in V$,
        \item $||v+w||\leq \max\{ ||v||,||w||\}$ for all $v,w\in V$.
    \end{itemize}
    When there is no risk of confusion, we simply refer to the $F$-norm $\nm$ as a \emph{norm} on $V$.
    Denote by $\calN(V)$ the space of norms on $V$.
    For any $c\in \R$, the function $$e^{-c}\nm:V\to\R_{\geq 0}$$ also satisfies the above conditions. Thus, $\calN(V)$ carries a natural $\R$-action. The quotient space $\calN(V)/\R$ is called the
    \emph{Goldman--Iwahori space} of $V$.
    A norm $\nm$ on $V$ is called \emph{diagonalizable} if there exists a basis $\{e_i\}$
    of $V$
    such that 
    $$||v||=\max_i\{|a_i|_F\cdot ||e_i||\}$$
    for any $v=\sum a_i e_i \in V,$ where $a_i\in F$. 
    Such a basis $\{e_i\}$ is called an \emph{orthogonal basis} for the norm $\nm$.
    Denote by $\calN^\diag (V)\subset \calN(V)$ 
    the set of diagonalizable norms on $V$. A diagonalizable norm on $V$ is called \emph{strictly Cartesian} if, moreover, the above orthogonal basis  $\{e_i\}$ satisfies $||e_i||=1$ for all $i$.
    Similarly, such a basis $\{e_i\}$ is called an \emph{orthonormal basis} for the norm $\nm$.
    Denote by $\calN^\str (V)\subset \calN^\diag (V)$ the set of strictly Cartesian norms on $V$.
\end{Def}

\begin{Rem}
    As pointed out in \cite[Remark~2.24]{RTW15}, it is known that $\calN^\diag (V)=\calN(V)$  if and only if $F$ is maximally complete.
    Note that our main base valued field $K:=\C((t))$ with the $t$-adic valuation is maximally complete, however, its completed algebraic closure $\C_K:=\widehat{\overline{K}}$ is not maximally complete; see \cite[Example 1.1]{BE21}.
\end{Rem}
\begin{Def}\label{df: continuous metric on nA spaces}

Let $F:=(F,\va_F)$ be a non-Archimedean field, and
consider a variety $X$ over $F$ with a line bundle $L$.
A \emph{(non-Archimedean) continuous metric $\phi$} on $L^\na$ is a family of norms
$$ \nm_{\phi (x)}: L(x):=L\otimes \scrH(x)\to \R_{\geq 0},\ \mathrm{for \ each\ } x\in X^\na,$$
where 
$\nm_{\phi(x)}$ is an $\scrH(x)$-norm on $L(x)$,  such that for any local trivializing section $s$ of $L$ on an open subset $U\subset X$, the induced function 
$$\phi_s:=-\log ||s||_\phi :U^\na\to \R $$
is continuous.
More explicitly, $\phi_s(x)=-\log ||s(x)||_{\phi(x)}$ for any $x\in U^\na$,
where $s(x)\in L(x)$ denotes the value of $s$ at $x$. 
\end{Def}
As we use additive notation for line bundles and metrics, if $\phi, \psi$ are continuous metrics on line bundles $L,M$, then
we denote by $\phi\pm\psi$ the induced continuous metric on $L\pm M$.
Furthermore, when $L=\calO_X$, a continuous metric $\phi$ is identified with a continuous function $\phi_1=-\log ||1||_\phi$ on $X^\na$.
Under this identification, the space of continuous metrics naturally carries an $\R$-action by addition of constants.
In particular, if $X$ is proper, then for any two continuous metrics $\phi$ and $\psi$ on $L^\na$,
the \emph{uniform norm} 
\begin{equation}\label{eq: uniform norm}
\sup_{x\in X^\na} |\phi-\psi|    :=   \sup_{ X^\na} \left|-\log||1||_{\phi-\psi}\right|
\end{equation}
 is well-defined.

\begin{example}
    \label{eg: log|s|}
    
For any section $s\in H^0(X,L)$, 
denote the non-vanishing locus of the section $s$ by
$X_s:=\{\mathfrak{p}\in X \ |\ s(\mathfrak{p})\neq 0\}$.
Then
$s$
induces a continuous metric
$\log |s|$ on $L^\na|_{X_s^\na}$ 
defined by
$$\log|s|_{s'}(x):=-\log |(s'/s)(x)|=\log |(s/s')(x)|$$ for any local trivializing section $s'$ of $L$ on
an open $U\subset X_s$ and any $x\in U^\na$, where
$|(s'/s)(x)|$ denotes the value at 
$x\in U^\na$ of the regular function
$s'/s\in \calO_X(U)$.
When $L=\calO_X$, $\log|s|$ is literally identified with a continuous function  of the form $\log|s(\cdot)|$
as $\log|s(\cdot)|=-\log|(1/s)(\cdot)|=\log|s|_1$.

Note that we may define $\log |s|$ even on the vanishing locus by allowing the induced function to take the value $-\infty$. Such a metric is called \emph{singular}, although all the metrics considered in this paper are continuous.
If we have a finite set of  sections $\{s_i\}\subset H^0(X,L)$ with no common zeros, then we obtain 
a continuous metric $\max_i \log|s_i|$ on $L^\na$ 
 given by 
$$ \left(\max_i \log|s_i| \right)_s(x):=\max\{ -\log |(s/s_i)(x)|\} $$
for any local trivializing section $s$ of $L$ on an open $U\subset X$.
Indeed, although
some of these functions
$-\log |(s/s_i)(x)|$ may take the value $-\infty$, at least one of them takes a value in $\R$.
Furthermore, given constants $\lambda_i\in \R$ for each $i$,
we may also define a continuous metric $\max_i \{ \log|s_i|+\lambda_i\}$ on $L^\na$ 
 in a similar manner.
\end{example}

In what follows, we mainly consider a geometrically connected smooth projective variety $X$ over $F$ with an ample line bundle $L$.
Since $L$ is ample, $kL$ is globally generated for all sufficiently large $k\in \Z_{>0}$.
Hence, $V_k:=H^0(X,kL)$ admits a finite set of sections  $\{s_i\}\subset H^0(X,kL)$ with no common zeros.
As observed in \pref{eg: log|s|}, given constants $\lambda_i\in \R$ for each $i$,
we have a continuous metric on $kL^\na$ of the form
$\max_i \{ \log|s_i|+\lambda_i\}$.
Since we have $s^k/s_i\in \calO_X(U)$ for any local trivializing section $s$ of $L$ on an open $U\subset X_{s_i}$, we obtain  a continuous metric $\frac{1}{k}\max_i \{ \log|s_i|+\lambda_i\}$ on $L^\na$
defined by 
\begin{equation}\label{eq: NAFS}
    \left(\frac{1}{k} \max_i\left\{\log|s_i| +\lambda_i\right\}\right)_s(x):=\frac{1}{k} \max_i\left\{-\log|(s^k/s_i)(x)| +\lambda_i\right\}
\end{equation}  
for any local trivializing section $s$ of $L$ on an open $U\subset X$.
\begin{Def} \label{df: NAFS and cpsh}
    A continuous metric on $L^\na$ of the form given in \eqref{eq: NAFS} is called a \emph{non-Archimedean Fubini--Study metric} (a \emph{NAFS metric} for short). Denote by $\FS(L^\na)$ the set of NAFS metrics on $L^\na$. Furthermore, a continuous metric $\phi$ on $L^\na$ is said to be \emph{continuous plurisubharmonic} (\emph{cpsh} for short) if $\phi$ can be written as the uniform limit of a decreasing net $\{\phi_i\}_i\subset \FS(L^\na)$. Denote by $\CPSH (L^\na)$ the set of cpsh metrics on $L^\na$, endowed with the topology of uniform convergence on $X^\na$ given by \eqref{eq: uniform norm}.
\end{Def}
By definition, we have a homeomorphism
$\FS(kL^\na)\simeq \FS(L^\na)$ for any $k\in \Z_{>0}$ defined
by $$\FS(kL^\na)\ni \phi \mapsto \frac{1}{k}\phi\in \FS(L^\na).$$
This homeomorphism extends to a homeomorphism $\CPSH(kL^\na)\simeq \CPSH(L^\na)$.
Moreover,  as in \cite[\S~3.7]{CLD12}, each $\phi\in \CPSH(L^\na)$ yields a positive Radon measure $\mathrm{MA}(\phi)$ on $X^\na$ of total mass $(L^d)$, where $d:=\dim X$.
This $\mathrm{MA}(\cdot)$ is called the \emph{Monge--Amp\`ere operator}.
More precisely, we have
\begin{equation}\label{eq: MA operator}
    \mathrm{MA}(\cdot):  \CPSH(L^\na)/\R\to \left\{\textrm{positive Radon measures on } X^\na \textrm{ of total mass } (L^{d}) \right\}.
\end{equation}
Here, the quotient $\CPSH(L^\na)/\R$ is defined with respect to the natural $\R$-action on $\CPSH(L^\na)$ by addition of constants, induced by the corresponding action on the space of continuous metrics.

The goal of this subsection is to study $\FS(L^\na)$ more explicitly, inspired by \cite{Don05}.
Motivated by the identification of the space of Hermitian norms on an $n$-dimensional complex vector space $V$ with
$\GL_n(\C)/U(n)$,
 we begin with the following observation:

\begin{Prop}\label{pr: structure theorem for strictly Cartesian norms}
    Let $V$ be an $n$-dimensional
    vector space over a non-Archimedean field $F:=(F,\va_F)$.
    Then we obtain the following bijection: $$\calN^\str (V)\simeq \GL_{n}(F)/\GL_n(F^\circ).$$
\end{Prop}
\begin{proof}
    By definition, we have a surjection $ \GL_n(F)\surj \calN^\str (V)$ given by
    $$\GL_n(F) \ni M_e:=(e_1 \cdots e_n)\longmapsto \nm_e:=\left(\sum_i a_i e_i \mapsto \max_i |a_i|_F \right) \in \calN^\str (V)$$
    for any $a_i\in F$,  where $e_i\in F^n$.
    If two bases $e:=\{ e_i\}$ and $f:=\{ f_i\}$ of $V$ induce the same Cartesian norm on $V$, the induced lattices $\Lambda_e:=\sum_i F^\circ e_i$ and  $\Lambda_f:=\sum_i F^\circ f_i$ coincide by \cite[Lemma~1.28]{BE21}, which implies $M_eM_f^{-1}\in \GL_n(F^\circ)$. Conversely, if $M_eM_f^{-1}\in \GL_n(F^\circ)$, then $M_e$ and $M_f$ induce the same lattice, which yields the same norm on $V$. Hence, the map $ \GL_n(F)\surj \calN^\str (V)$  induces the bijection $\GL_n(F)/\GL_n(F^\circ) \simeq \calN^\str (V)$.
\end{proof}

As introduced in the proof of \pref{pr: structure theorem for strictly Cartesian norms},
 we define the \emph{induced lattice} 
$V_{\nm}^\circ$ of $V$ for $\nm\in \calN^\str (V)$  by
\begin{equation}\label{eq: induced lattice}
V_{\nm}^\circ:=\{v\in V \ |\ ||v||\leq 1\}.    
\end{equation}
    In particular, this is a free $F^\circ$-module. Furthermore, we define
    \begin{equation}
        \label{eq: residual vector space}
    \widetilde{V}_\nm:=V_{\nm}^\circ \otimes_{F^\circ} \widetilde{F}.    
    \end{equation}

Recall that $X$ is assumed to be a geometrically connected smooth projective variety over $F$ with an ample line bundle $L$.
Assume that $L$ is globally generated.
Set $V:=H^0(X,L)$.
By \cite[Theorem 7.16]{BE21}, we may define the \emph{Fubini--Study map $\FS : \calN (V)\to \CPSH(L^\na)$} 
by
$$ \calN (V)\ni \nm \longmapsto \left(\FS(\nm)(x):= \log \sup_{s\in V\setminus\{0\}} \frac{|s(x)|}{||s||}\right)_{x\in X^\na}.$$
By definition, this is $\R$-equivariant.
Furthermore,
by \cite[Lemma~7.17]{BE21},
for any $\nm\in \calN^\diag (V)$ with an orthogonal basis $\{s_i\}$ such that $-\log||s_i||=\lambda_i$, the associated metric $\FS(\nm)\in \CPSH(L^\na)$ is given by
$$ \FS(\nm)=\max_i\{\log |s_i|+\lambda_i \}.$$
Hence, $\FS(\calN^\diag (V))\subset \FS(L^\na)$.

\begin{Def}
    
For each $k\in \Z_{>0}$, let
$V_k:=H^0(X,kL)$. We define
$$ \FS_k(L^\na):=\frac{1}{k}\FS(\calN^\diag (V_k))\subset \FS(L^\na),$$
 and call
 a metric in $\FS_k(L^\na)$ a \emph{NAFS metric on $L^\na$ at level $k$}.
In a similar manner, we also define
$$ \FS_k^\str (L^\na):=\frac{1}{k}\FS(\calN^\str (V_k))\subset \FS_k(L^\na),$$
and call
 a metric in $\FS_k^\str(L^\na)$ a \emph{strictly Cartesian NAFS metric on $L^\na$ at level $k$}.
\end{Def}

As in \eqref{eq: MA operator}, the differences arising from the $\R$-action on $\CPSH(L^\na)$ are essentially irrelevant to properties of metrics. Thus, we often consider the $\SL$-orbit in $\calN^\str(V_k)$, rather than the full $\GL$-orbit as in \pref{pr: structure theorem for strictly Cartesian norms}, as the locus of normalized norms. Indeed, we have the following:
\begin{Prop}
\label{pr: SL-orbit}
    If $F=\overline{F}$, then the bijection in \pref{pr: structure theorem for strictly Cartesian norms} implies   $$\SL_n(F)/\SL_n(F^\circ) \simeq \calN^\str(V)/\sim_\R,$$
    where the equivalence relation $\sim_\R$ on $\calN^\str(V)$ is
    induced by the $\R$-action on $\calN(V)$.
    In particular, we have a surjective map $\SL_n(F)/\SL_n(F^\circ)\surj \FS_1^\str(L^\na)/\R.$ 
\end{Prop}

\begin{proof}
    By \pref{pr: structure theorem for strictly Cartesian norms}, we have an inclusion
    $$ \SL_n(F)/\SL_n(F^\circ) \inj \GL_n(F)/\GL_n(F^\circ) \simeq \calN^\str(V).$$
    Since a nontrivial $\R$-equivalence on $\calN^\str(V)$ changes the corresponding determinant, the composition
    $$ \SL_n(F)/\SL_n(F^\circ) \inj \GL_n(F)/\GL_n(F^\circ) \simeq \calN^\str(V)\to \calN^\str(V)/\sim_\R$$ is injective. 
    Note that this $\R$-equivalence relation on $\calN^\str(V)$ does not define an $\R$-action on $\calN^\str(V)$ since an $\R$-orbit of $\nm\in \calN^\str(V)$ in $\calN(V)$ is not contained in $\calN^\str(V)$.
    Since we assume $F=\overline{F}$, for $M\in \GL_n(F)$, we may take $a:=(\det M)^{1/n}$. Then we have
    $[a^{-1}M]=[M]\in \calN^\str(V)/\sim_\R$ and $a^{-1}M\in \SL_n(F)$.
    Thus, we conclude the above composition is surjective and hence bijective.
\end{proof}
\begin{Rem}
    As in \cite[\S~10.2]{BT72}, the space $\calN^\diag(V)/\R$ coincides with the \emph{Bruhat--Tits building} 
    $\scrB(\SL_n, F)$
    for $\SL_n$ over $F$. Then the above $\SL_n(F)$-orbit $\SL_n(F)/\SL_n(F^\circ)$ in $\calN^\diag(V)/\R$ can be viewed as rational points of $\scrB(\SL_n, F)$.
    In particular, it is dense in $\scrB(\SL_n, F)$.
\end{Rem}

Let $F'/F$ be a complete field extension.
As in \cite[Lemma~6.1]{BE21}, we have a map $$(-)_{F'}: \CPSH(L^\na)\to \CPSH(L^\na_{F'})$$ 
induced by the pullback to $L^\na_{F'}$ via the base change $X_{F'}\to X$.

\begin{Prop}
    \label{pr: injectivity of the base change}
    The above map $(-)_{F'}: \CPSH(L^\na)\to \CPSH(L^\na_{F'})$ is injective.
\end{Prop}
\begin{proof}
  It follows from the fact that the base change $X_{F'}\to X$ is surjective. 
\end{proof}

\begin{Def}\label{df: alg NAFS}
 Keep the same situation as in \pref{pr: injectivity of the base change}.
    We define 
    the \emph{class
     $$\FS_k^\str(L^\na_{F'})|_L:=\{\phi\in 
     \FS_k(L^\na) \ |\ \phi_{F'} \in  \FS_k^\str(L^\na_{F'})\}$$ of NAFS metrics on $L^\na$ at level $k$ inducing elements of $\FS_k^\str(L^\na_{F'})$}. In particular,
     for the \emph{completed algebraic closure} $\C_F:=\widehat{\overline{F}}$ of $F$ equipped with the valuation $\va_{\C_F}$ extending $\va_F$, we call 
     \begin{equation} \label{eq: algebraic NAFS}
     \phi\in \FS_k^\alg(L^\na):=\bigcup_{\substack{F\subset F'\subset \overline{F}\\ [F':F]<\infty}} \FS_k^\str(L^\na_{F'})|_L\subset \FS_k^\str(L^\na_{\C_F})|_L    
     \end{equation}
     an \emph{algebraic NAFS metric on $L^\na$ at level $k$}.
\end{Def}

\begin{example}\label{eg: level k NAFS from strictly Cartesian norms}
Typical elements of $\FS_k^\str(L^\na_{F'})|_L$ are given as follows:
Suppose that $V_k=H^0(X,kL)$ admits a basis $\{s_i\}$ with no common zeros, and take $\lambda_i\in \log|F'^\times|$ for each $i$. 
Then a diagonalizable norm $\nm\in \calN^\diag(V_k)$ with the orthogonal basis $\{s_i\}$ such that $-\log ||s_i||=\lambda_i$ yields 
$$\frac{1}{k}\FS(\nm)=\frac{1}{k}\max_i\{\log |s_i|+\lambda_i \}\in \FS_k^\str(L^\na_{F'})|_L.$$
Indeed,
take $a_i\in F'^\times$ with $\lambda_i=\log|a_i|$ for each $i$. Then $\{a_i s_i\}$ forms a basis of 
$V_{k,F'}:=H^0(X,kL_{F'})$. 
Since $\log |a_is_i|=\log |s_i|+\lambda_i$, 
the strictly Cartesian norm $\nm'$ on $V_{k, F'}$ with the orthonormal basis $\{a_i s_i\}$ yields $\frac{1}{k}\FS(\nm')=\frac{1}{k}\FS(\nm)_{F'}$.
Hence we have $\frac{1}{k}\FS(\nm)\in \FS_k^\str(L^\na_{F'})|_L$.
Namely, if $\nm\in \calN^\diag(V_k)$ induces  $\nm_{F'}\in \calN^\str(V_{k, F'})$, then $\frac{1}{k}\FS(\nm)\in \FS_k^\str(L^\na_{F'})|_L$.

\end{example}
\begin{example}\label{eg: algebraic NAFS and model}
Consider
a model $(\scrX,\scrL)$ of $(X_{\overline{F}},kL_{\overline{F}})$ over 
$\overline{F}^\circ:=\overline{F}\cap\C_F^\circ$,
namely, a flat projective $\overline{F}^\circ$-scheme $\calX$ equipped with an isomorphism
$\scrX_{\overline{F}}\simeq X_{\overline{F}}$ and a line bundle $\scrL$ on $\scrX$ such that 
$\scrL_{\overline{F}}\simeq kL_{\overline{F}}$
 under this isomorphism.
Suppose that $\scrL$ is relatively very ample.
Then, any $\overline{F}^\circ$-basis of $H^0(\scrX,\scrL)$ yields 
 a strictly Cartesian norm $\nm$ on $V_k\otimes \C_F\simeq H^0(\scrX,\scrL)\otimes \C_F$ by considering the basis as an orthonormal basis of $V_k\otimes \C_F$ for $\nm$.
  Since $V_{k}$ is finite dimensional, the basis of $V_k\otimes \C_F$ descends to a basis of $V_k\otimes F'$ for some 
 finite extension $F'$ of $F$.
 Thus, the norm  $\nm$ can be seen as a strictly Cartesian norm $\nm$ on $V_k\otimes F'$.
 Furthermore, by \pref{pr: structure theorem for strictly Cartesian norms}, the norm $\nm$ is independent of the choice of an $\overline{F}^\circ$-basis of $H^0(\scrX,\scrL)$.
Thus the model $(\scrX,\scrL)$ induces a NAFS metric 
$\varphi_{\scrL}:=\frac{1}{k}\FS(\nm)_{\C_F}\in \FS_k^\str(L^\na_{\C_F})$.
Furthermore, we have
\begin{equation}\label{eq: algebraic metrics over F}
\varphi_{\scrL}\in \calM:= \bigcup_{\substack{F\subset F'\subset \overline{F}\\ [F':F]<\infty}} \FS_k^\str(L^\na_{F'})_{\C_F} \subset \FS_k^\str(L^\na_{\C_F}).     
\end{equation}
Conversely, if $\varphi\in \calM$, then we have $\varphi=\frac{1}{k}\FS(\nm)$ for some $\nm\in \calN^\str(V_{k}\otimes F')$, where $F' \subset \overline{F}$ is a finite extension of $F$.
Then, by using 
an orthonormal basis for $\nm$, we obtain an isomorphism 
$$ f:\PP((V_k\otimes \overline{F})^\vee)\simeq \PP^n_{\overline{F}},$$
where $n=\dim V_k-1$.
Since $f$ induces an open immersion $ \PP((V_k\otimes \overline{F})^\vee)\simeq \PP^n_{\overline{F}} \inj  \PP^n_{\overline{F}^\circ}$ into $ \PP^n_{\overline{F}^\circ}$,
we obtain a model $\scrX$ of $X_{\overline{F}}$ as a scheme-theoretic closure of  $X_{\overline{F}}$ in $\PP^n_{\overline{F}^\circ}$.
Furthermore, setting $$\scrL:=f^*\calO_{\PP^n_{\overline{F}^\circ}}(1)|_\scrX,$$
we  see that
 $(\scrX,\scrL)$ is
a model of $(X_{\overline{F}},kL_{\overline{F}})$ over 
$\overline{F}^\circ$ such that $\scrL$ is relatively very ample.
In particular, by construction, $(\scrX,\scrL)$ induces the original norm $\nm\in \calN^\str(V_{k}\otimes F')$, which implies
 $$\varphi=\varphi_\scrL.$$
Therefore, for any $\phi\in \FS_k^\alg(L^\na)$, the pullback $\phi_{\C_F}\in \calM$ is obtained in this way.
That is,
there exists a model  $(\scrX,\scrL)$ of $(X_{\overline{F}},kL_{\overline{F}})$
    over $\overline{F}^\circ$ such that $\phi_{\C_F}=\varphi_{\scrL}$.

\end{example}

To describe norms  corresponding to \eqref{eq: algebraic metrics over F}, we introduce the following:
\begin{Def}\label{df: alg norm}
Use the same notation as in \pref{df: alg NAFS}.
    Let $V_F$ be an $F$-vector space of $\dim V_F<\infty$.
    Set $V:=V_F\otimes \C_F$.
    Then we define 
    $$ \calN^\alg_F(V):= \bigcup_{\substack{F\subset F'\subset \overline{F}\\ [F':F]<\infty}} \calN^\str(V_F\otimes F')_{\C_F}\subset \calN^\str (V)    $$
    and call $\nm\in \calN^\alg_F(V)$ an
    \emph{algebraic norm on $V$ over $F$}.
\end{Def}
\begin{Rem} Note that
    $ \calN^\alg_F(V)$ is defined for $V_F$. 
    Suppose $V_F=H^0(X,kL)$. Then,
    for any  $\nm\in \calN^\alg_F(V)$, we have $\frac{1}{k}\FS(\nm) \in \calM,$ where
    $\calM$ is the one given by \eqref{eq: algebraic metrics over F}.
\end{Rem}
\subsection{Critical metrics}\label{ss: critical metric}
The aim of this subsection is to recall the notion of critical metrics over non-Archimedean fields.
Originally, a critical metric was introduced in \cite{Zha96} as a special Fubini--Study metric for a polarized complex manifold $(X,L)$, where $L$ is a very ample line bundle on $X$. 
More precisely, it is a Fubini--Study metric 
induced by a closed embedding $X\inj  \PP(V^\vee)\simeq \PP_\C^{n},$
where $V:=H^0(X,L)$ and $n:=h^0(X,L)-1$, such that the associated Chow form $R_X$ minimizes the Chow norm $\nm_\mathrm{Ch}$ under the special linear action $\SL (V)$.
As Zhang's approach is purely algebro-geometric, we may apply it to the setting of non-Archimedean fields, so that we obtain a special NAFS metric induced by an optimal embedding into a projective space with respect to the Chow norm.
Named after Zhang's work, the NAFS metric is also called a \emph{critical metric} in \cite{Fan22}.

In what follows, 
we consider the setting of non-Archimedean fields more precisely.
Let $X$ be a projective variety with a very ample line bundle $L$ over a non-trivially valued non-Archimedean field $F$ of characteristic zero such that $F=\overline{F}$.
Consider a closed embedding $\iota:X\inj \PP(V^\vee)$ induced by the complete linear system $|L|$, where $V^\vee$ is the dual vector space of $V:=H^0(X,L)$. Set $d:=\dim X$ and $\delta:=\deg \iota=(L^d)$.
Then this gives
a point $[X]$ in the Chow variety $\mathrm{Chow}_{d,\delta} (\PP (V^\vee))$ parameterizing $d$-dimensional cycles of degree $\delta$ in $\PP(V^\vee)$.
Here, we may identify each closed point in $\PP(V)$ with the corresponding  hyperplane in $\PP(V^\vee)$.
Observe that
$$\left\{ (H_0,\dots, H_d)\in \PP(V)^{d+1} \ \middle| \ H_0\cap\cdots \cap H_d \cap \iota(X) \neq \emptyset \textrm{ in } \PP(V^\vee)  \right\}$$
is a divisor in $\PP(V)^{d+1}$ of multi-degree $(\delta,\dots, \delta)$.
Let $$R_X\in   W:=\left(\mathrm{Sym}^{\delta} (V^\vee)\right)^{\otimes (d+1)}$$
denote the defining equation of this divisor. 
This $R_X$ is called the \emph{Chow form of $X$}.
By definition, $R_X$ is unique up to scalar dilation.
Hence, the induced class $$[R_X]\in \PP\left( W\right)$$
is uniquely determined, and called the \emph{Chow point of $X$}.
Furthermore, as in \cite[Chapter~4, Theorem~1.1]{GKZ94}, 
a map 
\begin{equation}\label{eq: Chow embedding+}
    \mathrm{Chow}_{d,\delta} (\PP (V^\vee))\to \PP\left( W\right) ; [X]\mapsto [R_X]
\end{equation}
factors as the composition of the following two morphisms:
$$\mathrm{Chow}_{d,\delta} (\PP(V^\vee)) \inj \PP\left(H^0\left(\mathrm{Gr}(d+1,V), \calO(\delta)\right)\right) \inj \PP\left( W\right).$$
The first morphism  $\mathrm{Chow}_{d,\delta} (\PP(V^\vee)) \inj \PP\left(H^0\left(\mathrm{Gr}(d+1,V), \calO(\delta)\right)\right)$ 
is
the \emph{Chow embedding},
where $\mathrm{Gr}(d+1,V)$ is the Grassmannian variety parametrizing $(d+1)$-dimensional linear subspaces in $V$, and 
$\calO(1)$ on $\mathrm{Gr}(d+1,V)$ is induced by the Pl\"ucker embedding.
The second morphism is induced by 
$$H^0\left(\mathrm{Gr}(d+1,V), \calO(\delta)\right) \inj  W.$$
In particular, the map \eqref{eq: Chow embedding+} is a morphism.
Furthermore, by construction, this morphism is equivariant under the special linear action $\SL (V)$. 

\begin{Def}
    In the above setting, the polarized variety $(X,L)$ over $F$ is said to be \emph{Chow stable} (resp. \emph{Chow polystable}, \emph{Chow semistable}) if the Chow point $[R_X]\in \PP(W)$ is GIT stable (resp. polystable, semistable) with respect to the action of $\SL(V)$. Here, recall that $[R_X]$ is \emph{GIT semistable} if the closure of the orbit 
    $$\SL(V)\cdot R_X \subset W$$  does not contain the origin $0\in W$, \emph{GIT polystable} if this orbit is closed in $W$, and \emph{GIT stable} if, in addition, its stabilizer is finite.
    Then $(X,L)$ is  said to be \emph{asymptotically Chow stable}  if $(X,kL)$ is Chow stable for all sufficiently large $k\in \Z_{>0}$. 
\end{Def}
A strictly Cartesian norm $\nm\in \calN^\str (V)$ on $V$ induces a dual norm $\nm^\vee\in \calN^\str (V^\vee)$ on $V^\vee$. Furthermore, the dual norm $\nm^\vee\in \calN^\str (V^\vee)$ induces a strictly Cartesian norm
$$\nm_\mathrm{Ch}\in \calN^\str \left(W\right)$$
on $W=\left(\mathrm{Sym}^{\delta} (V^\vee)\right)^{\otimes (d+1)}$, which is called the \emph{Chow norm}.
Note that it induces a strictly Cartesian NAFS metric on $\calO_{\PP(W)}(1)^\na$ through a dual norm.
In analogy with the Kempf--Ness work \cite{KN79}, we may consider a continuous function 
$p_X: \SL (V)^\na \to \R$ defined by
$$p_X(g):=||g^*R_X||_\mathrm{Ch}.$$
Here, we regard $\SL (V)$ as an affine group scheme over $F$, and consider its analytification 
$\SL (V)^\na$.
Indeed, by construction, both
the action of $\SL (V)$ on the morphism \eqref{eq: Chow embedding+} and the Chow norm
$\nm_\mathrm{Ch}$ are compatible with valued field extension. Hence, for any $g\in \SL (V)^\na$,
the value $p_X(g)=||g^*R_X||_\mathrm{Ch}$ makes sense after base change to the completed residue field $\scrH(g)$ at $g\in  \SL (V)^\na$.
Properties of the function $p_X$ are studied as part of a non-Archimedean version of the Kempf--Ness theory, which first appeared in \cite{Bur92} and was further developed in \cite[\S~3]{Mac17}.

\begin{Rem}
    In \cite{Fan22}, our function $p_X$ was denoted by $\mu_X$. However, in \cite{Zha96}, $\mu(g,[X])$ was introduced as
    $\mu(g,[X])=\log p_X(g)-\log p_X(e)$.
    To avoid confusion and to maintain consistency with the complex setting, we use the notation $p_X$ instead of Fang’s notation.
\end{Rem}

Denote by $U_{\nm}$ the subgroup of $\SL (V)^\na$ preserving the strictly Cartesian norm $\nm$ on $V$ as introduced in \cite[Theorem~2.1]{RTW10}.
In addition, denote by $\SL (V_{\nm}^\circ)$ the reductive $F^\circ$-group scheme of special linear actions on $V$ inducing an isomorphism $V_{\nm}^\circ\simeq V_{\nm}^\circ$, where
$V_{\nm}^\circ$ is the induced lattice defined in \eqref{eq: induced lattice}.
 In particular,
 an orthonormal basis for $\nm$ induces an isomorphism  $\SL (V_{\nm}^\circ)\simeq \SL_{n}\otimes_\Z \Spec F^\circ$,
 where $n:=\dim V$.
 Then we have
 $$U_{\nm}\simeq \SL (V_{\nm}^\circ)^\beth,$$
 where $\SL (V_{\nm}^\circ)^\beth$ is the analytification of $\SL (V_{\nm}^\circ)$
 in the sense of Thuillier; see \cite{Thu07}.
 Thus, 
 $U_{\nm}$ is a \emph{maximal compact subgroup} in the sense of \cite[\S~3, Definition~2.21]{Mac17}.
In the same way, we obtain a maximal compact subgroup $U_{\nm_\mathrm{Ch}}$ of 
$\SL(V)^\na$
with respect to 
the Chow norm $\nm_\mathrm{Ch}$ on $W$.

\begin{Prop}\label{pr: NA unitary}
Keep the same notation as above. Then we have 
     $U_{\nm}= U_{\nm_{\mathrm{Ch}}}$.
\end{Prop}
\begin{proof}
By construction, we have $U_{\nm}\subset U_{\nm_{\mathrm{Ch}}}$.
Since $U_{\nm}$ is maximal among the compact subgroups of $\SL(V)^\na$,
we obtain $U_{\nm}= U_{\nm_{\mathrm{Ch}}}$.
\end{proof}

In particular,
the function $p_X$ is independent of the choice of an orthonormal basis for the strictly Cartesian norm $\nm\in \calN^\str (V)$. This ensures that the following notion is well-defined:

\begin{Def}\label{df: NA critical metric}
    A NAFS metric $\varphi\in \FS_1^\str(L^\na)$ associated to a strictly Cartesian norm $\nm\in \calN^\str(V)$ is said to be \emph{critical} if the function $p_X$ attains its minimum at the identity element $e:=\id$ of $\SL(V)^\na$.
    Similarly, a NAFS metric $\varphi\in \FS_k^\str(L^\na)$ at level $k$ associated to a strictly Cartesian norm $\nm\in \calN^\str(V_k)$, where $V_k:=H^0(X,kL)$, is also said to be \emph{critical at level $k$}  if the function $p_X$, defined with $V_k$ in place of $V$, attains its minimum at the identity element $e=\id$ of $\SL(V_k)^\na$.
\end{Def}

\begin{Rem}\label{rm: well-definedness of critical metric}
    Since $R_X$ is determined only up to scalar multiplication, \pref{df: NA critical metric} is independent of the choice of $R_X$.
\end{Rem}

\begin{Prop}
\label{pr: F-valued pts are dense}
Keep the notation as in \pref{df: NA critical metric}.
    The following are equivalent:
    \begin{itemize}
        \item the NAFS metric $\varphi\in \FS_1^\str(L^\na)$ is critical (at level $1$),
        \item the function $p_X$ attains its minimum at the identity element  of $\SL(V)^\na$ over $\SL(V)(F)$.
    \end{itemize}
\end{Prop}
\begin{proof}
Since $p_X$ is continuous on $\SL(V)^\na$, it is enough to show that $\SL(V)(F)\subset \SL(V)^\na$ is dense.
By \cite[2.1.15.~Proposition]{Ber90}, this follows from  the assumption that $F$ is algebraically closed 
and non-trivially valued.
\end{proof}

\subsection{Totally degenerate abelian varieties}
\label{ss: Totally degenerate abelian varieties}

Set $T:=\Spec \Z [M]$ and $T_K:=T\times_\Z \Spec K,$
where $M$ is a free $\Z$-module.
An abelian variety $X$ over $K:=\C((t))$ with a $t$-adic valuation $\va_K$ such that $|t|_K:=e^{-1}$ is said to be \emph{totally degenerate} if its analytification $X^\na$ admits a morphism
$p: T_K^\na \surj X^\na$
of analytic groups such that 
$$
\Gamma_X:=\ker p \subset T_K^\na (K)\simeq \Hom (M,K^\times)    
$$
is a lattice, that is, 
$\Gamma_X$ is a free $\Z$-module of 
$\rk \Gamma_X=\rk M$, and
it induces an isomorphism $$T_K^\na / \Gamma_X\simeq X^\na$$
of $K$-analytic spaces.
In particular, $\Gamma_X$ acts on $T_K^\na$ by translation. 
Set $N:= \Hom (M,\Z)$.
The \emph{tropicalization map}
$$\trop_K: T_K^\na \to N_\R:=N\otimes_\Z\R$$ 
is given by
$ x\longmapsto (m \mapsto -\log |z^m(x)|)_{m\in M}$ for any $x\in T_K^\na$.
Then it admits a canonical section $$\mathfrak{G}:N_\R \inj T_K^\na$$ defined 
by assigning to each 
$n\in N_\R$ the corresponding monomial valuation $\frakG(n)\in T_K^\na$, that is,
$$
\frakG(n):
\sum_{m\in M} a_m z^m \longmapsto \max_{a_m\neq 0}\left\{ |a_m|_K \cdot e^{-\langle m,n\rangle} \right\},
$$
where $\langle -,-\rangle: M_\R\times N_\R\to \R$ is the canonical pairing.
The section $\frakG : N_\R\to T_K^\na$ is called the \emph{Gauss section}.
From this observation, we have the \emph{canonical (Berkovich) retraction} 
$$
\rho_X:=\overline{\trop_K}: X^\na \simeq T_K^\na / \Gamma_X \surj N_\R/\trop_K (\Gamma_X)
$$
with the
 canonical section $\frakG:N_\R/\trop_K (\Gamma_X) \inj X^\na$.
Indeed,
$\rho_X$ can be seen as a deformation retraction onto
$\mathrm{Sk}(X):=\frakG (N_\R/\trop_K (\Gamma_X))$, which
is called the \emph{canonical skeleton}. 
Through this identification, we often identify $\mathrm{Sk}(X)$ with $N_\R/\trop_K (\Gamma_X)$.
Note that it coincides with the essential skeleton of $X$; see, for example, \cite[Proposition~4.3.2]{HN17}.

Given an ample line bundle $L$ on  $X$, its pullback
$p^*(L^\na)$ of $p:T_K^\na \surj X^\na$ is trivial; see, for example, \cite[Theorem~6.3.3]{FvdP}.
After fixing a cubical structure on $L$ and a trivialization of $p^*(L^\na)$,
the ample line bundle $L$ gives rise to
the data 
\begin{equation} \label{eq: data for L}
    (\Lambda, a, b,\Phi)
\end{equation}
consisting of a free $\Z$-module $\Lambda$, 
a map
$ a: \Lambda \to K^\times,$
a bimultiplicative map $b:\Lambda\times M\to K^\times,$
and an injective homomorphism $\Phi: \Lambda \to M$ with finite cokernel,
satisfying the following relations:
\begin{enumerate}
   \item  $a(0)=1,$ and 
    for any $\lambda, \lambda'\in \Lambda,$
   \begin{equation}\label{eq: 2-cocycle}
     a(\lambda+\lambda')=b(\lambda,\Phi(\lambda'))a(\lambda)a(\lambda').
   \end{equation}
    \item $q(\lambda,\lambda'):=b(\lambda, \Phi (\lambda'))$
    is a $K^\times$-valued symmetric pairing on $\Lambda \times \Lambda$,
    namely,
$$
        q(\lambda,\lambda')=q(\lambda',\lambda)
 $$  
    for any  $\lambda, \lambda'\in \Lambda.$ Besides,
   the  $\R$-valued symmetric pairing
    $-\log |q(\lambda,\lambda')|_K$ is positive definite.
\end{enumerate}
This set of data first appeared in \cite{Mum72} and was studied further in greater generality in \cite{FC}, where it is referred to as \emph{degeneration data}.
In particular, the data $(\Lambda, a, b,\Phi)$ together with $M$
recover the original polarized abelian variety $(X,L)$ by the \emph{Mumford construction};
see \cite{Mum72}.
By the Mumford construction, the data $(\Lambda, a, b,\Phi)$ can be interpreted as follows:
\begin{enumerate}

     \item The natural injection $\beta: \Lambda \inj \Hom (M,K^\times)\simeq T_K^\na(K)$ induced by $b:\Lambda \times M\to K^\times$  satisfies
    \begin{equation}\label{eq: period matrix}
        \beta(\Lambda)=\Gamma_X.
    \end{equation}
    In particular, we have $z^m(\beta(\lambda))=b(\lambda,m)$ for any $m\in M$ and for any $\lambda\in \Lambda$.

    \item Let  $S_{\beta(\lambda)}$ denote
    the translation on $T_K^\na$ by 
    $\beta(\lambda)$ for each $\lambda \in \Lambda$.
    Under
    the trivialization of $p^*(L^\na)$ and the cubical structure of $L$,
    the identity $p\circ S_{\beta(\lambda)}=p$ induces an isomorphism 
    \begin{equation}\label{eq: c_lambda}
    c_\lambda :S_{\beta(\lambda)}^* \calO_{T_K^\na}\xrightarrow{\sim} \calO_{T_K^\na},    
    \end{equation}
    and this $c_\lambda$ is 
     given by multiplication by the \emph{factor of automorphy} $j_\lambda (z):=a(\lambda)z^{\Phi(\lambda)}.$
    In particular, the composition $c_\lambda\circ S_{\beta(\lambda)}^*: \calO_{T_K^\na} \simeq \calO_{T_K^\na}$ 
    induces an action $$\tau_\lambda: H^0(T_K^\na,\calO_{T_K^\na}) \to H^0(T_K^\na,\calO_{T_K^\na}),$$ defined by
    \begin{equation}\label{eq: factor of automorphy}
       z^m\longmapsto a(\lambda)z^{\Phi(\lambda)}S_{\beta(\lambda)}^*(z^m)=a(\lambda)b(\lambda,m)z^{m+\Phi(\lambda)}
    \end{equation} 
    for 
    any $z^m\in  H^0(T_K^\na,\calO_{T_K^\na})$, such that 
    \begin{equation}\label{eq: periodic condition on global sections}
        H^0(X^\na,L^\na)\simeq  H^0(T_K^\na,\calO_{T_K^\na}) ^\Lambda :=\left\{f \in H^0(T_K^\na,\calO_{T_K^\na}) \ \middle| \ \tau_{\lambda}(f)=f ,\ ^\forall \lambda \in \Lambda \right\}.
    \end{equation} 
    Furthermore, it follows, for instance,  from \cite[Theorem~6.4.4]{Lut16} that 
    \begin{equation}\label{eq: h^0(L)}
        h^0(X^\na,L^\na)=\sharp(\coker \Phi).
    \end{equation}

\end{enumerate}
The above observation shows that $a,b$ and $\Phi$ depend only on $L$, on $X$, and on the class $[L]\in \mathrm{Pic} (X)$, respectively.
Furthermore, 
setting
$g:=\dim X=\rk M=\rk \Lambda$,
 the Smith normal form of 
$\Phi:\Lambda\inj M$ 
yields bases 
$\{\lambda_i\}_{1\leq i\leq g}\subset \Lambda$  and  $\{m_i\}_{1\leq i\leq g}\subset M $
   such that 
      \begin{equation}\label{eq: Smith normal form}
   \Phi(\lambda_i)=e_im_i
     \end{equation}
   for some positive integers $e_1, \dots, e_g$ satisfying
       $e_i|e_{i+1}$
  for each $1\leq i<g$.
In particular,
with respect to these bases, we have
$\Lambda \simeq \Z^g$ and $T_K\simeq \mathbb{G}_{m,K}^g$.
Moreover,
the homomorphism
$b^*:\Lambda\to \Gamma_X\subset T_K^\na(K)$ is explicitly described by
 the period matrix $$P=(p_{ij})_{i,j}:=(b(\lambda_j,m_i))_{i,j}=(\beta(\lambda_1)\cdots \beta(\lambda_g))\in \GL_{g}(K).$$ 
Thus, setting $P^{\Z^g}:=\langle \beta(\lambda_1),\cdots, \beta(\lambda_g)\rangle \subset (K^\times)^g \simeq (\mathbb{G}_{m,K}^g)^\na(K)$, we obtain an isomorphism
\begin{equation}\label{eq: Mumford--Tate uniformization}
    X^\na\simeq T_K^\na / \beta(\Lambda) \simeq (\mathbb{G}_{m,K}^g)^\na /P^{\Z^g},
\end{equation}
which is called the \emph{Mumford uniformization} of $(X^\na,L^\na)$.
  Here, observe that
  $$q(\lambda_j,\lambda_i)=b(\lambda_j,\Phi(\lambda_i))=b(\lambda_j,e_im_i)=p_{ij}^{e_i}.$$
  Compare \eqref{eq: Mumford--Tate uniformization} with \cite[Lemma~2.6]{GO24}.

\begin{Prop}\label{pr: explicit form of a}
    After replacing the trivialization of $p^*(L^\na)$,
    for all $\lambda \in \Lambda$, we have
$$
    a(\lambda)^2=q(\lambda,\lambda).
$$
\end{Prop}
\begin{proof}
    Define $c(\lambda):=a(\lambda)^2q(\lambda,\lambda)^{-1}$ for all $\lambda\in \Lambda$.
    If $c:\Lambda \to K^\times$ is a homomorphism, then we can take a trivialization of $p^*(L^\na)$ such that
$c\equiv 1$, which implies $a(\lambda)^2=q(\lambda,\lambda)$. 

From now on, we show that $c$ is a homomorphism.
By \eqref{eq: 2-cocycle} and the symmetry of $q$, we have
\begin{align*}
    c(\lambda+\lambda')&=a(\lambda+\lambda')^2q(\lambda+\lambda',\lambda+\lambda')^{-1} \\
    &= q(\lambda,\lambda')^2 a(\lambda)^2a(\lambda')^2 q(\lambda, -\lambda-\lambda') q(\lambda', -\lambda-\lambda') \\
    &= a(\lambda)^2 q(\lambda, -\lambda) a(\lambda')^2 q(\lambda', -\lambda') \\
    &= c(\lambda) c(\lambda').
\end{align*}
    Hence,  $c:\Lambda \to K^\times$ is a homomorphism.
\end{proof}
Throughout this paper, we always take the trivialization of $p^* (L^\na)$ as in \pref{pr: explicit form of a}.

For $k\in \Z_{>0}$, the cubical structure on $L$ and the trivialization of $p^* (L^\na)$
that we fixed induce corresponding structures on $kL$.
Furthermore, we have the following:
\begin{Prop}\label{pr: data for kL}
For the induced cubical structure on $kL$ and the induced trivialization of $p^* (kL^\na)$,
the data associated to $kL$ are given by 
 $$\left(\Lambda, a^k, b, k\Phi\right),$$ 
 where 
     $(\Lambda, a, b,\Phi)$ is the data associated to $L$.
\end{Prop}
\begin{proof}
    Denote by $\left(\Lambda_k, a_k, b_k,\Phi_k\right)$ the data  associated to $kL$.
    By \eqref{eq: period matrix}, we may identify $\Lambda_k$ with $\Lambda$ and $b_k$ with $b$.
    Then it follows from \eqref{eq: factor of automorphy} that 
    $a_k(\lambda)=a(\lambda)^k$ and $\Phi_k(\lambda)=k\Phi(\lambda)$.
\end{proof}

\section{Non-Archimedean balanced metrics on totally degenerate abelian varieties}
\label{sc:NA balanced on totally degenerate abelian varieties}
In this section, motivated by the work of Donaldson \cite{Don01}, we introduce the notion of a \emph{non-Archimedean balanced metric (at level $k$)} as a variant of Fang's critical metric recalled in \pref{ss: critical metric}, and study such metrics, particularly on polarized totally degenerate  abelian varieties. 

Before introducing the notion of a non-Archimedean balanced metric, we briefly recall some important properties of critical metrics in the complex setting.
Here, note that critical metrics are also called \emph{balanced metrics}, particularly from a differential-geometric viewpoint.
First, as proved in \cite[Theorem~3.2]{Zha96}, if $(X,L)$ is Chow stable, then such a metric exists uniquely.
Second, as proved in \cite[Theorem~3]{Don01}, if $\Aut(X,L)$ is discrete and $X$ admits a constant scalar curvature K\"ahler metric $\omega_\infty$ (a cscK metric for short) in the first Chern class $c_1(L)$, then $(X,kL)$ admits a unique balanced metric inducing $\omega_k\in c_1(L)$ for all sufficiently large $k\in\Z_{>0}$, and 
we have
$$ \lim_{k\to\infty}\omega_k=\omega_\infty $$
in $\mathscr{C}^\infty$. 
Furthermore, by \cite[Theorem~0.1]{Luo98}, this implies that $(X,L)$ is asymptotically Chow stable, namely, $(X,kL)$ is Chow stable for all sufficiently large $k\in\Z_{>0}$.
In particular,  these assumptions are satisfied when $(X,L)$ is a polarized Calabi--Yau manifold (cf. the proof of \pref{pr: asymptotically Chow stable}).

\subsection{Non-Archimedean balanced metrics}
\label{ss: NA balanced metrics}

Consider a polarized projective variety $(X,L)$ over $K:=\C((t))$ with a $t$-adic valuation such that $|t|_K:=e^{-1}$.
Let $$\CK:=\widehat{\overline{K}}$$ be a \emph{completed algebraic closure  of $K$} equipped with the valuation $\va_\CK$ extending $\va_K$.
Here, as in \cite[Example~1.1]{BE21}, the non-Archimedean field $\CK$ is realized as the field of formal series $$f=\sum_{r\in \Q}a_rt^r$$
with support $\supp f:=\{r\in \Q \ |\ a_r\neq 0\}$, where $a_r\in \C$, containing only finitely many elements with a given upper bound.
Since $\CK$ satisfies all the assumptions imposed on the field $F$ in \pref{ss: critical metric},
we can apply the discussion in \pref{ss: critical metric} to the case $F=\CK$. From now on, we exploit the specific features of this setting to study the following non-Archimedean balanced metrics:

\begin{Def}\label{df: NA balanced}
Let $(X,L)$ be given as above.
    A NAFS metric $\phi_k\in \FS_k (L^\na)$ is said to be \emph{non-Archimedean balanced} (\emph{NA balanced} for short) at level $k$ if 
    $\phi_k$ is an algebraic NAFS metric at level $k$ in the sense of \eqref{eq: algebraic NAFS} such that 
    $(\phi_k)_\CK\in \FS_k^\str (L_\CK^\na)$ is 
    critical at level $k$.
\end{Def}

\begin{Rem}\label{rm: NA balanced is independent of finite base change}
    Arguments concerning algebraic varieties over $K$ are often carried out 
    up to finite base change. One reason for considering algebraic NAFS metrics is to avoid the subtleties arising from such base changes. Indeed, \pref{df: NA balanced} is independent of the choice of finite base change.

\end{Rem}

Suppose that $L$ is very ample on $X$ and 
$X$ can be seen as a $d$-dimensional cycle in $\PP(V_K^\vee)$ of degree $\delta:=(L^d)$ by
the induced linear system $|L|:X_\CK\inj \PP(V_K^\vee),$ where  $V_K:=H^0(X,L)$.
As in \cite{Kol96}, although $K$ is not algebraically closed, we can define the Chow point 
$$[R_X]\in \PP(W_K)(K)$$ of $|L|:X\inj \PP(V_K^\vee)$, where 
$W_K:=\left(\mathrm{Sym}^{\delta} (V_K^\vee)\right)^{\otimes (d+1)}$,
and consider its GIT stability. Setting $$V:=V_K\otimes \CK\simeq H^0(X_\CK,L_\CK),$$
consider an algebraic norm $\nm\in \calN^\alg_K(V)$ over $K$ (cf.~\pref{df: alg norm}).
That is, for some finite extension $K'$ of $K$, the norm
 $\nm$ is a strictly Cartesian norm on $V_{K'}:=V_K\otimes K'$. 
Here, since $K=\C((t))$, we may assume  $K'= \C ((t^{1/s}))$ for some $s\in \Z_{>0}$.
In particular, $K'$ is also a discretely valued field.
Since $\nm\in \calN^\str(V_{K'})$,
the induced Chow norm $\nm_\mathrm{Ch}$ on $W_{K'}:=\left(\mathrm{Sym}^{\delta} (V_{K'}^\vee)\right)^{\otimes (d+1)}$ is a strictly
Cartesian norm on $W_{K'}$. Therefore, it
yields a lattice $(W_{K'})_{\nm_\mathrm{Ch}}^\circ$ and a model 
$$ \left(\scrY,\calO_\scrY(1)\right):=\left(\PP\left((W_{K'})_{\nm_\mathrm{Ch}}^\circ\right), \calO_{\PP\left((W_{K'})_{\nm_\mathrm{Ch}}^\circ\right)}(1) \right)$$
of $ \left(\PP\left(W_{K'}\right), \calO_{\PP\left(W_{K'}\right)}(1) \right)$ over $(K')^\circ$ as in \pref{eg: algebraic NAFS and model}, where
$\PP\left((W_{K'})_{\nm_\mathrm{Ch}}^\circ\right)$ is the projectivized bundle
for
$(W_{K'})_{\nm_\mathrm{Ch}}^\circ$.
By \cite[Theorem~2.1]{RTW10} (cf.~\pref{ss: critical metric}),
a $K'$-analytic space $\SL(V_{K'})^\na$ has a maximal compact subgroup $U'_\nm$ preserving the strictly Cartesian norm $\nm$ on $V_{K'}$ such that
$$ U'_\nm\simeq (\mathscr{G}_\nm')^\beth, \text{ where } \mathscr{G}_\nm' :=  \SL\left((V_{K'})_\nm^\circ\right).$$
By the same argument  as in \pref{pr: NA unitary},
$\mathscr{G}_\nm'$ acts on $\scrY$,
endowing
$\calO_{\scrY}(1)$
with the structure of an ample  $\mathscr{G}_\nm'$-linearlized line bundle.
Since the Chow point $[R_X]\in \PP(W_K)(K)$ induces
$$ [R_{X_{K'}}]\in \PP\left(W_{K'}\right) (K'),$$
 by the valuative criterion for properness, we have a unique lift
 $$[R_{X_{K'}}]^\circ\in \scrY((K')^\circ)$$ of $[R_{X_{K'}}]\in \PP\left(W_{K'}\right)(K')$.
 Then we define the \emph{residual Chow point}
$$[\widetilde{R_{X_{K'}}}]
\in \left(\scrY\times_{\Spec (K')^\circ} \Spec\widetilde{K'}\right)(\widetilde{K'}) 
\simeq \PP \left(\widetilde{(W_{K'})}_{\nm_\mathrm{Ch}}\right)(\widetilde{K'})$$ 
by its reduction to $\Spec \widetilde{K'}$, where
$\widetilde{(W_{K'})}_{\nm_\mathrm{Ch}}$ is defined for 
$\nm_\mathrm{Ch}\in \calN^\str(W_{K'})$
by \eqref{eq: residual vector space}.
Here,
$\PP\left(\widetilde{(W_{K'})}_{\nm_\mathrm{Ch}}\right)$ is a projective space over $\widetilde{K'}$, equipped with an action of the $\widetilde{K'}$-reductive group $$\widetilde{U'_\nm}:=\mathscr{G}_\nm'\times_{\Spec (K')^\circ} \Spec\widetilde{K'}.$$
Thus, we can consider the GIT stability 
of $[\widetilde{R_{X_{K'}}}]\in \PP \left(\widetilde{(W_{K'})}_{\nm_\mathrm{Ch}}\right)(\widetilde{K'})$ for $\widetilde{U'_\nm}$ again.
Note that, we can also define
$[R_{X_\CK}]\in \PP(W)$ and $[\widetilde{R_{X_\CK}}]\in \PP \left(\widetilde{W}_{\nm_\mathrm{Ch}}\right)$, where
$W:=\left(\mathrm{Sym}^{\delta} (V^\vee)\right)^{\otimes (d+1)}$,
for $\nm\in \calN_K^\alg(V)\subset \calN^\str(V)$ by viewing $\nm$ as a strictly Cartesian norm on $V$, and consider their GIT stabilities for
$\mathscr{G}_\nm:=\SL(V_\nm^\circ)$
and $\widetilde{U_\nm}:=\mathscr{G}_\nm\times_{\Spec \C_K^\circ} \Spec\widetilde{\CK}$ respectively.
Observe that, since they are given by the base change along the valued field extension $\CK /K'$, 
the GIT stability of 
$[R_{X_{\CK}}]$ (resp. 
$[\widetilde{R_{X_{\CK}}}]$)
with respect to $\SL(V)$ (resp. 
$\widetilde{U_\nm}$ ) 
is equivalent to 
the GIT stability of $[R_{X_{K'}}]$ (resp. 
$[\widetilde{R_{X_{K'}}}]$) with respect to 
$\SL(V_{K'})$ (resp. 
$\widetilde{U'_\nm}$).
In particular, since $(K')^\circ=\C[[t^{1/s}]]$ is universally japanese (cf.~\cite[Corollaire (7.7.4)]{EGA4II}),  we can apply \cite[Theorem~2]{Ses77}, which plays a crucial role in
 \cite[Chapter~3, Theorem~4.9]{Mac17}. 
 More precisely, combining this with \pref{pr: F-valued pts are dense},
we have the following:
\begin{Thm}[{cf.~\cite[Chapter~3, Theorem~4.9]{Mac17}}]

\label{th: Maculan's thm}
Keep the same notation as above.
Suppose that $[R_{X_\CK}]\in \PP(W)$ is GIT semistable with respect to $\SL(V)$. Then, for any  $\nm\in \calN^\alg_K(V)$,
the following are equivalent:
\begin{itemize}
    \item the induced NAFS metric $\varphi=\FS(\nm)\in \FS_1^\str(L_\CK^\na)$ is critical (at level $1$),
    \item  the induced residual Chow point $[\widetilde{R_{X_\CK}}]\in \PP \left(\widetilde{W}_{\nm_\mathrm{Ch}}\right)$  is GIT semistable with respect to $\widetilde{U_\nm}$.
\end{itemize}
    
\end{Thm}
\begin{Def}\label{df: residual Chow stability}
    Let $(X,L)$ be a polarized variety over $K=\C((t))$.
    The polarized variety $(X,L)$ is said to be \emph{residually Chow stable} (resp. \emph{polystable}, \emph{semistable}) if
    there exists an algebraic norm $\nm$ on $V:=H^0(X,L)\otimes \CK$ over $K$ such that the induced residual Chow point $$[\widetilde{R_{X_\CK}}]\in \PP \left(\widetilde{W}_{\nm_{\mathrm{Ch}}}\right)$$  is GIT stable (resp. polystable, semistable) with respect to $\widetilde{U_{\nm}}$.
    Furthermore, $(X,L)$ is said to be \emph{asymptotically residually Chow stable} (resp. \emph{polystable}, \emph{semistable}) if $(X,kL)$ is residually Chow stable (resp. polystable, semistable) for all sufficiently large $k\in \Z_{>0}$.
\end{Def}
 Since $K$ is discretely valued, we can apply the semistable reduction theorem in
 \cite[Theorem~2.7]{Zha96} to our situation. 
 Since,  as in \pref{eg: algebraic NAFS and model}, such a semistable reduction
 gives rise  to an algebraic norm whose induced residual Chow point is GIT semistable,
 we conclude the following:
\begin{Thm}[{cf.~\cite[Theorem~2.7]{Zha96}}]
\label{th: Zhang's semistable reduction}
If a polarized variety $(X,L)$ over $K=\C((t))$ is  Chow semistable, then $(X,L)$ is residually Chow semistable. 
In particular, there exists a critical metric 
 $\varphi\in \FS_1^\str(L_\CK^\na)$ on $L_\CK^\na$
(at level $1$).
\end{Thm}
\begin{Rem}
In contrast to \pref{th: Zhang's semistable reduction}, Chow semistability of $(X,L)$ does not, in general, imply the existence of a critical metric in the complex setting.
More precisely, the existence of a critical metric in the complex setting is equivalent to Chow polystability rather than mere Chow semistability.
We also
note that the GIT stability of the residual Chow point associated with $\nm\in \calN^\str(V)$ determines whether the norm $\nm$ actually induces a minimizer of 
$p_{X_\CK} : \SL(V)^\na \to \R,$ whereas the GIT stability of the Chow point associated with $(X,L)$ determines whether such a norm 
$\nm$
exists.

\end{Rem}

\begin{Prop}\label{pr: res Chow st eqv to res Chow st over CK}
    Keep the same notation as in \pref{th: Maculan's thm}.
    If $(X,L)$  is Chow semistable, then
    the following are equivalent:
    \begin{enumerate}
    \item[$\mathrm{(i)}$] $(X,L)$ is residually Chow stable,
        \item[$\mathrm{(ii)}$] there exists a strictly Cartesian norm $\nm\in \calN^\str(V)$ such that the induced residual Chow point $[\widetilde{R_{X_\CK}}]\in \PP \left(\widetilde{W}_{\nm_{\CK,\mathrm{Ch}}}\right)$  is GIT stable.

    \end{enumerate}
    Furthermore, if these equivalent conditions hold, then the norm $\nm\in\calN^\str(V)$ given in $\mathrm{(ii)}$ is unique up to an $\R$-equivalence. In particular, the associated NAFS metric $\varphi:=\FS(\nm)\in \FS_1^\str(L_\CK^\na)$ is a unique critical metric on $L_\CK^\na$ at level $1$, up to an additive constant.
\end{Prop}
\begin{proof}
By definition, (i) implies (ii). Thus, it suffices to show that (ii) implies (i).
Recall that $V=H^0(X_\CK,L_\CK)$ and $W=\left(\mathrm{Sym}^{\delta} (V^\vee)\right)^{\otimes (d+1)}$.
Set $n:=\dim V$.
For any field extension $K'$ such that $K\subset K'\subset \C_K$,
set $V_{K'}:=H^0(X_{K'},L_{K'})$.

By fixing a 
strictly Cartesian metric $\nm_0$ on $V_K$, we have the following inclusion, after identifying $\calN^\str(V_{K'})$ with its image under the natural injection
$
\calN^\str(V_{K'})\inj \calN^\str(V)$:
$$\SL(V_{K})(K)\nm_0\subset \SL(V_{K'})(K')(\nm_0)_{K'}\subset  \SL(V)(\CK)(\nm_0)_{\CK}\subset \calN^\str(V),$$
where $(\nm_0)_{K'}$ is a norm on $V_{K'}$ induced by $\nm_0$.
Note that, by fixing $\nm_0$, the ambiguity arising from the 
$\R$-equivalence
disappears, as in \pref{pr: SL-orbit}.
Set
$$[\SL(V_{K'})]:=\SL(V_{K'})(K')(\nm_0)_{K'}\simeq \SL_n(K')/\SL_n((K')^\circ)
\subset \scrB(\SL_n,K')\simeq \calN^\diag(V_{K'})/\R .$$
Then we define the stable (resp. semistable) reduction locus in the orbit $[\SL(V_{K'})]\subset \calN^\str(V_{K'})$ with respect to
${[R_{X_\CK}]}\in \PP(W)$ 
as follows:
$$\mathrm{SRL}_{[R_{X_\CK}]}^{K'}:=\left\{ \nm_g\in [\SL(V_{K'})] \ \middle| \ 
[\widetilde{R_{X_\CK}}]\in \PP \left(\widetilde{W}_{\nm_{g,\mathrm{Ch}}}\right)
\text{ by }
\nm_g \text{ is GIT stable}
\right\}$$
$$\left( \text{resp. } \mathrm{SSRL}_{[R_{X_\CK}]}^{K'}:=\left\{ \nm_g\in [\SL(V_{K'})] \ \middle| \ 
[\widetilde{R_{X_\CK}}]\in \PP \left(\widetilde{W}_{\nm_{g,\mathrm{Ch}}}\right)
\text{ by }
\nm_g \text{ is GIT semistable}
\right\} \right).$$
Since we assume $\nm\in \mathrm{SRL}_{[R_{X_\CK}]}^{\CK}$ in (ii),
by \cite[Theorem 1]{GO26},  we have
\begin{equation}\label{eq: GO26}
  \mathrm{SRL}_{[R_{X_\CK}]}^\CK=\mathrm{SSRL}_{[R_{X_\CK}]}^\CK=\{ \nm \}.  
\end{equation}
Since $(X,L)$ is Chow stable,  by \pref{th: Zhang's semistable reduction}, we have a strictly Cartesian norm $\nm'\in[\SL(V_{K'})]$ on $V_{K'}$ for some finite extension $K'$ of $K$ such that $\nm'\in \mathrm{SSRL}_{[R_{X_\CK}]}^{K'}$.
Therefore, we obtain $$\nm_\CK'=\nm\in \mathrm{SRL}_{[R_{X_\CK}]}^\CK.$$
That is, $\nm'$ induces a residual Chow point that is GIT stable, which implies (i).
Furthermore, the uniqueness of such a norm follows from \eqref{eq: GO26}.
\end{proof}

\begin{Thm}\label{th: res Chow stable implies a unique NA balanced metric}
    If a polarized variety $(X,L)$ over $K=\C((t))$ is residually Chow stable, then there exists a unique NA balanced metric 
 on $L^\na$
(at level $1$), up to an additive constant.
\end{Thm}
\begin{proof}
Use the same notation as in the proof of \pref{pr: asymptotically Chow stable}.
Since we assume that $(X,L)$ is residually Chow stable, 
the equality \eqref{eq: GO26} remains valid.
In particular, it
yields
$\mathrm{SRL}_{[R_{X_\CK}]}^{K'}
=\{ 1 \text{pt}\}$
for some finite extension $K'$ of $K$.
Thus, we obtain a unique strictly Cartesian norm $$\nm'\in [\SL(V_{K'})]\inj \scrB(\SL_n,K')$$ such that $\FS(\nm')_\CK$ is critical.
Here, observe that 
any finite extension $K'$ of $K$ is isomorphism to $\C((t^{1/s}))$ for some $s\in \Z_{>0}$.
In particular, the extension $K'/K$ is a tamely ramified Galois extension.
Then,
the Galois group $\mathrm{Gal}(K'/K)$ acts on $\mathrm{SRL}_{[R_{X_\CK}]}^{K'}$. 
Indeed,
each Galois action $\sigma\in \mathrm{Gal}(K'/K)$ 
sends a 
stable reduction of $X_{K'}$ to its $\sigma$-conjugate, which is again a 
stable reduction of $X_{K'}$.
Since
$\mathrm{SRL}_{[R_{X_\CK}]}^{K'}
=\{ 1 \text{pt}\}$, 
the norm $\nm'\in \mathrm{SRL}_{[R_{X_\CK}]}^{K'}$ is $\mathrm{Gal}(K'/K)$-invariant.
Here, since $K'/K$ is tamely ramified,
it follows from \cite[Proposition~5.1.1]{Rou77} (see also \cite{Pra01}) that
$$ \scrB(\SL_n,K')^{\mathrm{Gal}(K'/K)}\simeq \scrB(\SL_n,K) \left( \simeq \calN^\diag(V_{K}) /\R\right).
$$
Therefore, 
there exists a unique diagonalizable norm
$\nm_K\in \calN^\diag(V_{K})$ on $V_K$ such that 
\begin{equation}\label{eq: scale fixing}
    \nm'=(\nm_K)_{K'}.
\end{equation}
Note that  the ambiguity arising
from the $\R$-equivalence on $\calN^\diag(V_{K})$ disappears by \eqref{eq: scale fixing}. 
Then $\FS(\nm_K)_\CK=\FS(\nm')_\CK$ is critical, which means that $\FS(\nm_K)\in \FS_1(L^\na)$ is NA balanced. Moreover, the uniqueness follows from its construction.
\end{proof}
\begin{Cor}\label{cr: asym res Chow st yields a sequence of NA balanced}
   If a polarized variety $(X,L)$ over $K=\C((t))$ is asymptotically residually Chow stable, then, 
   for every sufficiently large $k\in \Z_{>0}$, there exists a unique NA balanced metric $\phi_k$ on $L^\na$ at level $k$, up to an additive constant.
\end{Cor}

\subsection{Theta functions and symmetries}
\label{ss: theta functions and symmetries}

Keep the notation from \pref{ss: Totally degenerate abelian varieties}, and consider the polarized totally degenerate abelian variety $(X,L)$ over $K$ with the associated data  $(\Lambda, a,b,\Phi)$.
Since $\Phi : \Lambda \inj M$ is injective, we may define a bimultiplicative map $b_\Phi: \Phi (\Lambda)\times M\to K^\times$
 by $$b_\Phi (\Phi (\lambda),m):=b(\lambda,m)$$ for $\lambda\in \Lambda$ and $m\in M$.
Then the valuation $\va_K$ extends uniquely to a valuation $\va_{\overline{K}}$ on $\overline{K}$.
 Since $\overline{K}^\times$ is  divisible, we obtain a (non-canonical) extension $$b_\Phi:M\times M\to \overline{K}^\times$$ that is bimultiplicative.
 Furthermore, we may assume that the extension $b_\Phi:M\times M\to \overline{K}^\times$ is symmetric
as $b_\Phi:\Phi(\Lambda)\times \Phi (\Lambda)\to K^\times$ is symmetric.
By \eqref{eq: 2-cocycle} and
\pref{pr: explicit form of a},
we also obtain an extension $a_\Phi: M\to \overline{K}^\times$ such that, 
 for  $m,m'\in M$ and $\lambda\in \Lambda$,
 \begin{itemize}
     \item  $a_\Phi(\Phi(\lambda))=a(\lambda),$
     \item $ a_\Phi(m+m')=a_\Phi(m)a_\Phi(m')b_\Phi(m,m'),$
     \item $a_\Phi(m)^2=b_\Phi(m,m)$.
 \end{itemize}
 Note that
 the ambiguity in extending $b_\Phi$ arises from the choice of an $\ell$-th root of an element of $K^\times$ for $\ell\in \Z_{>0}$.
 Indeed,
 for any $m\in M$, there exists $r\in \Z_{>0}$ and $\lambda\in \Lambda$ such that $rm=\Phi(\lambda)$, which implies that $b_\Phi (m,m')^r=b_\Phi (rm,m')=b(\lambda,m')$ for any $m'\in M$.
Hence, although $b_\Phi:M\times M\to \overline{K}^\times$ itself is not canonical,
the induced bilinear form
$B_\Phi:M\times M\to \Q$, defined by
$B_\Phi(m,m'):=-\log |b_\Phi(m,m')|_{\overline{K}}$ for $m,m'\in M$,
is canonical.
Note that $B_\Phi$ is positive definite since the bilinear form $-\log |q(\cdot,\cdot)|_K$ is positive definite.
Furthermore, we may consider a unique extension $B_\Phi:M_\R\times M_\R\to \R$, 
and define $A_\Phi:M_\R\to \R$ by 
$$
    A_\Phi(m):=\frac{1}{2}B_\Phi(m,m)
$$
 for $m\in M_\R$.
In particular, we have $A_\Phi (m)=-\log |a_\Phi(m)|_{\overline{K}}$ for $m\in M$.

In what follows, using $a_\Phi$ and $b_\Phi$, we study
$H^0(X^\na,L^\na)\simeq H^0(X,L)$ and
$$H^0(X^\na,L^\na) \hat{\otimes} \C_K \simeq H^0(X_{\C_K}^\na, L_{\C_K}^\na)\simeq H^0(X_{\C_K},L_{\C_K}), $$
where $\C_K:=\widehat{\overline{K}}$ denotes the completed algebraic closure of $K$ as in \pref{ss: critical metric}. Note that these isomorphisms follow from a GAGA-type argument; see \cite[Corollary~3.4.10]{Ber90}.
The base change allows us to
choose a more convenient basis for the Fourier series expansion of $H^0(T_\CK^\na, \calO_{T_\CK^\na}) 
$. Namely, we may
 replace the basis $\{z^m\}_{m\in M}$ with $$\{\tilde{z}^m\}_{m\in M}:=\{a_\Phi(m)z^m\}_{m\in M}.$$
With respect to the new basis $\{\tilde{z}^m\}_{m\in M}$, the action described in \eqref{eq: factor of automorphy} 
is simply given by
\begin{equation}\label{eq: action of lambda}
    \tau_\lambda (\tilde{z}^m)=\tilde{z}^{m+\Phi(\lambda)}.
\end{equation}
Indeed, we have $\tau_\lambda (\tilde{z}^m)=a_\Phi(m)a_\Phi(\Phi(\lambda))b_\Phi(m,\Phi(\lambda))z^{m+\Phi(\lambda)}=a_\Phi({m+\Phi(\lambda)})z^{m+\Phi(\lambda)}=\tilde{z}^{m+\Phi(\lambda)}.$
Moreover,
we may define an action of $M$ on
$H^0(T_\CK^\na, \calO_{T_\CK^\na}) 
$  by
\begin{equation}
    \label{eq: maximal level subgroup}
\sigma_m(\tilde{z}^{m'}):=\tilde{z}^{m+m'}    
\end{equation}
 for all $m, m'\in M$.

 Let $S_x$ denote translation by $x$ on $ T^\na_\CK$ (resp. $X^\na_\CK$), where $x\in T_\CK (\CK)$ (resp. $x\in X_\CK (\CK)$).
 Then, the action \eqref{eq: maximal level subgroup} can be written as
 $$
\sigma_m(z^{m'})=a_\Phi(m) z^m S_{\beta_\Phi(m)}^*z^{m'},     
 $$
where $\beta_\Phi(m) \in  T_\CK (\CK)=T_\CK^\na (\CK)$ is defined by
$z^{m'}(\beta_\Phi(m))=b_\Phi(m,m')$.
Here, we define
$$\calG(\calO_{T_\CK^\na}):=\left\{(x,c)\ |\ x\in T_\CK (\CK), \ c: S_{x}^*\calO_{T_\CK^\na} \xrightarrow{\sim} \calO_{T_\CK^\na}\right\}$$
with the  group law given by 
$$(x,c)\cdot (x',c'):=(x+x',c\circ S_{x}^*c').
$$
Then a map $\rho : M\to \calG(\calO_{T_\CK^\na})$ given by $\rho(m):=(\beta_\Phi(m), a_\Phi(m) z^m)$ is a homomorphism. Indeed, 
\begin{align*}
\rho(m)\cdot \rho(m')
&=(\beta_\Phi(m), a_\Phi(m)z^m)\cdot (\beta_\Phi({m'}), a_\Phi (m')z^{m'}) \\
&=(\beta_\Phi(m)+\beta_\Phi(m'), a_\Phi(m)z^m \circ S_{\beta_\Phi(m)}^* \circ a_\Phi (m')z^{m'} \circ S_{-\beta_\Phi(m)}^*)    \\
&=(\beta_\Phi(m+m'), a_\Phi(m)z^m a_\Phi (m') b_\Phi (m,m') z^{m'} S_{0}^*)\\
&=(\beta_\Phi(m+m'), a_\Phi(m+m')z^{m+m'})=\rho(m+m').
\end{align*}
Furthermore, the group $\calG(\calO_{T_\CK^\na})$ carries 
 a natural representation on  $H^0(T_\CK^\na, \calO_{T_\CK^\na})$ given by 
 $$(x,c)\mapsto c\circ S_x^*.$$
 In particular, via this natural representation,
the above action $\sigma : M\to \Aut (H^0(T_\CK^\na, \calO_{T_\CK^\na}))$ factors through
$\rho : M\to \calG(\calO_{T_\CK^\na})$.

Since $\rho(\Phi(\lambda))=(\beta(\lambda),a(\lambda)z^{\Phi(\lambda)})=(\beta(\lambda),c_\lambda)$ for  $\lambda\in \Lambda$,
where $c_\lambda$ is the one given in \eqref{eq: c_lambda},
we have $\sigma_{\Phi(\lambda)}=\tau_\lambda$
for $\lambda\in \Lambda$.
Furthermore, 
 setting $\rho(m)=(x_m,c_m)$ for each $m\in M$, the equality
$\rho(m)\rho(\Phi(\lambda))=\rho(\Phi(\lambda))\rho(m)$ implies that
the isomorphism $$c_m: S_{x_m}^* \calO_{T_\CK^\na} \xrightarrow{\sim} \calO_{T_\CK^\na}$$ is 
$\Lambda$-equivariant. 
Namely, $c_m\circ S_m^*c_\lambda =c_\lambda\circ S_{\beta(\lambda)}^*c_m$.
Here, note that
 the trivialization of $p^*L^\na$ induces a trivialization of $p_\CK^* L_\CK^\na$, where 
$p_\CK: T_\CK^\na\surj X_\CK^\na$ is the covering map induced by $p:T_K^\na\surj X^\na$.
Via the induced trivialization of $p_\CK^* L_\CK^\na$, the $\Lambda$-equivariant isomorphism
$c_m$ descends to an isomorphism 
$$\bar{c}_m: S_{\bar{x}_m}^* L_\CK \xrightarrow{\sim} L_\CK,$$
where $\bar{x}_m:=p_\CK(x_m)$.
Thus,
the pair $(\bar{x}_m,\bar{c}_m)$ defines 
 an element of 
the \emph{theta group}
$$\mathcal{G}(L_\CK):=\{(x,c)\ |\ x\in X (\CK), \ c: S_{x}^*L_\CK\xrightarrow{\sim} L_\CK\}.$$
Furthermore,
since $\bar{c}_{\Phi(\lambda)}:L_\CK =S_0^*L_\CK \xrightarrow{\sim} L_\CK$ 
is trivial for every $\lambda\in \Lambda$,
the  homomorphisms $\rho : M\to \calG(\calO_{T_\CK^\na})$ and $\sigma : M\to \Aut (H^0(T_\CK^\na, \calO_{T_\CK^\na}))$ descend to homomorphisms
 $$\bar{\rho} : \coker \Phi \to \mathcal{G}(L_\CK) \textrm{ and } \bar{\sigma}: \coker \Phi \to \Aut (H^0(X_\CK, L_\CK)),$$ respectively.
In particular, we note that
 the theta group $\mathcal{G}(L_\CK)$ carries 
 a natural representation on  $H^0(X_\CK, L_\CK)$ given by 
 $(x,c)\mapsto c\circ S_x^*,$
and $\bar{\sigma}$ factors through $\bar{\rho}$ via the representation of $\mathcal{G}(L_\CK)$.

Here, we consider the following \emph{classical theta function}:
$$
\theta:=\sum_{\lambda\in \Lambda} \tilde{z}^{\Phi(\lambda)}\in H^0(T_\CK^\na, \calO_{T_\CK^\na}).
$$
More explicitly, we have
$\theta=\sum_{\lambda\in \Lambda} a(\lambda)z^{\Phi(\lambda)}\in H^0(T_K^\na, \calO_{T_K^\na})$.
By \eqref{eq: action of lambda}, the equality $\tau_\lambda \theta=\theta$ holds for each $\lambda\in \Lambda$. Thus, it follows from \eqref{eq: periodic condition on global sections} that $$\theta\in H^0(X^\na,L^\na)\simeq H^0(X,L) \inj H^0(X_{\C_K},L_{\C_K}).$$
\begin{Def}\label{df: normalized theta with characteristic}
    For $\mu\in \coker \Phi$, we define the \emph{normalized theta function with characteristic $\mu$} by
$$
    \vartheta_\mu:=\bar{\sigma}_\mu \theta\in H^0(X_{\C_K},L_{\C_K}).
$$
\end{Def}
Then it follows from \eqref{eq: h^0(L)} that  
$\{\vartheta_\mu\}_{\mu\in \coker \Phi}$ forms a basis of $H^0(X_{\C_K},L_{\C_K})$.
 
For the projection $\mathcal{G}(L_\CK)\to K(L_\CK):=\{x\in X_\CK(\CK) \ |\ {}^\exists c: S_{x}^*L_\CK\xrightarrow{\sim} L_\CK\},$
the composition 
$$ \coker \Phi  \xrightarrow{\bar{\rho} }\mathcal{G}(L_\CK) \to  K(L_\CK)$$
is injective.
Furthermore, by  \cite[\S~16]{Mum70}, we have
$$\dim K(L_\CK)=(h^0(X_\CK,L_\CK))^2=(\sharp(\coker \Phi))^2.$$
Thus, it follows from \cite[Theorem~1]{Mum66} that 
$$(\coker \Phi)^\sim:= \bar{\rho} ( \coker \Phi) \subset \mathcal{G}(L_\CK)$$
is a maximal level subgroup of 
$\mathcal{G}(L_\CK)$.
Here, by \cite[Theorem~2]{Mum66}, 
 $H^0(X_{\C_K},L_{\C_K})$ is an irreducible representation of $\mathcal{G}(L_\CK)$.
 Since $\bar{\sigma}_{\mu+\nu}=\bar{\sigma}_\mu\bar{\sigma}_\nu$ for $\mu, \nu\in \coker \Phi$,
we have 
$$
\bar{\sigma}_{\nu}\vartheta_\mu=\vartheta_{\mu+\nu}.    
$$
Moreover, by \cite[Theorem~3]{Mum66}, with respect to this basis,
 the representation of
$\mathcal{G}(L_\CK)$ on $H^0(X_{\C_K},L_{\C_K})$ is identified with the Schrödinger representation of
a \emph{Heisenberg group} $\calH$, which fits into the short exact sequence
\begin{equation}\label{eq: ses for theta group}
    1\to \C_K^\times \to \calH \to \coker \Phi \times \Hom (\coker \Phi, \C_K^\times) \to 1.
\end{equation} 
In particular, via the Schrödinger representation, the character group $\Hom (\coker \Phi, \C_K^\times)$ admits a canonical lift to 
$\calG(L_\CK)$, whose representation on $H^0(X_{\C_K},L_{\C_K})$ is given by $$\chi(\vartheta_\mu):=\chi(\mu)\vartheta_\mu$$ for each $\chi\in \Hom (\coker \Phi, \C_K^\times)$ and $\mu\in \coker \Phi$.
Then the image  of this lift
$$(\Hom (\coker \Phi, \C_K^\times))^\sim$$
is another maximal level subgroup  of $\mathcal{G}(L_\CK)$. 
Observe that
$$\chi \overline{\sigma}_\mu=\chi(\mu)\overline{\sigma}_\mu \chi$$ for $\chi\in \Hom (\coker \Phi, \C_K^\times)$ and $\mu\in \coker \Phi$.
In particular,
$\chi(\mu)\in  \C_K^\times$ is a root of unity since $\coker \Phi$ is finite. 
More precisely, setting
$\ell:=e_g$ for the Smith normal form of $\Phi$ (see \eqref{eq: Smith normal form}), we have
$$\chi(\mu)\in \mu_\ell:=\{\zeta\in \C_K^\times \ |\ \zeta^\ell=1\}=\{\zeta\in \C^\times \ |\ \zeta^\ell=1\}.$$
Here, we consider a subgroup $H$ of $\mathcal{G}(L_\CK) \simeq  \calH $ defined by
$$ H:=\langle (\coker \Phi)^\sim, (\Hom (\coker \Phi, \C_K^\times))^\sim\rangle \subset \mathcal{G}(L_\CK) ,$$
and call it the associated \emph{finite Heisenberg group}.
Observe that
the sequence
 \eqref{eq: ses for theta group} induces 
$$
  1\to \mu_\ell \to H \to \coker \Phi \times \Hom (\coker \Phi, \C_K^\times) \to 1.
  $$
Since $\C_K^\times \subset \calH$ acts on 
$H^0(X_{\C_K},L_{\C_K})$ only by scalars, it does not affect the irreducibility of the representation. Thus, 
$H^0(X_{\C_K},L_{\C_K})$ is also an irreducible representation of $H$.

Given a lift $m_\mu\in M$ of $\mu$, we have
$$
    \vartheta_\mu=\sum_{\lambda\in \Lambda} \tilde{z}^{m_\mu+\Phi(\lambda)}
    =\sum_{\lambda\in \Lambda} a_\Phi(m_\mu+\Phi(\lambda)) z^{m_\mu+\Phi(\lambda)}= a_\Phi(m_\mu)\sum_{\lambda\in \Lambda} a(\lambda)b(\lambda,m_\mu)z^{m_\mu+\Phi(\lambda)}.
$$
\begin{Def}\label{df: theta with characteristic}
    For the lift $m_\mu$, we define the \emph{theta function with characteristic $\mu$} by
    $$\theta_{m_\mu}:=\sum_{\lambda\in \Lambda} a(\lambda)b(\lambda,m_\mu)z^{m_\mu+\Phi(\lambda)}\in H^0(X^\na,L^\na)\simeq H^0(X,L).$$

\end{Def}
\begin{Rem}
    Note that $\theta_{m_\mu}$ depends on the choice of the lift $m_\mu$, while $\vartheta_\mu$ does not.
\end{Rem}
In what follows, we fix a lift $m_\mu$ for each $\mu$, and write $\theta_{m_\mu}$ simply as  $\theta_\mu$. 
Since we have $$\vartheta_\mu=a_\Phi(m_\mu)\theta_\mu,$$
the collection $\{\theta_\mu \}_{\mu\in \coker \Phi}$ also forms a basis of $H^0(X_{\C_K},L_{\C_K})$, and hence of $H^0(X,L)$.

Since the above discussion only uses the data $(\Lambda,a,b,\Phi)$,
 we can apply the same discussion to the data $(\Lambda,a^k,b,k\Phi)$ that corresponds to $kL$ by
  \pref{pr: data for kL},
 where $k\in \Z_{>0}$.
 That is, we may consider extensions $a_{k\Phi}^k: M \to \overline{K}^\times$ and $b_{k\Phi}:M\times M\to \overline{K}^\times$ such that,
 for  $m,m'\in M$ and $\lambda\in \Lambda$,
 \begin{itemize}
     \item  $a_{k\Phi}^k(k\Phi(\lambda))=a(\lambda)^k, \ b_{k\Phi}(k\Phi(\lambda),m)=b(\lambda,m),$
     \item $ a_{k\Phi}^k(m+m')=a_{k\Phi}^k(m)a_{k\Phi}^k(m')b_{k\Phi}(m,m'),$
     \item $a_{k\Phi}^k(m)^2=b_{k\Phi}(m,m)$.
 \end{itemize}
 These extensions allow us to take a convenient basis $\{a_{k\Phi}^k(m)z^m\}_{m\in M}$ for the Fourier series expansion of $H^0(T_\CK^\na, \calO_{T_\CK^\na})$, which induces a natural action $\bar{\sigma}$ on $H^0(X_{\C_K},kL_{\C_K})$ by
$\coker k\Phi$.

\begin{Def}\label{df: normalized theta at level k}
    We define
 the \emph{normalized theta function $\vartheta_\mu^{(k)}\in  H^0(X_{\C_K},kL_{\C_K})$ at level $k$ with characteristic $\mu\in \coker (k\Phi)$} by
    $$\vartheta_\mu^{(k)} := \bar{\sigma}_\mu \theta^{(k)}\in H^0(X_{\C_K},kL_{\C_K}),$$
where $\theta^{(k)}$ is the \emph{classical theta function at level $k$} given by
$$ \theta^{(k)}:= \sum_{\lambda\in \Lambda} a(\lambda)^k z^{k\Phi(\lambda)}\in H^0(X,kL).$$
\end{Def}
Then
$\{\vartheta_\mu^{(k)}\}_{\mu\in \coker k \Phi}$ forms a basis of  $H^0(X_{\C_K},kL_{\C_K})$.
Furthermore, 
the natural representation of the theta group $\calG(kL_\CK)$ on $H^0(X_{\C_K},kL_{\C_K})$ is identified with 
the Schrödinger representation of a Heisenberg group 
$\calH_k$  fitting into the short exact sequence
$$
    1\to \C_K^\times \to \calH_k \to \coker k\Phi \times \Hom (\coker k\Phi, \C_K^\times) \to 1.
$$
More precisely, via the canonical lifts of 
$\coker k\Phi$ and $\Hom (\coker k\Phi, \C_K^\times)$ to $\calH_k\simeq \calG(kL_\CK)$, 
these groups act on  $H^0(X_{\C_K},kL_{\C_K})$ by
\begin{equation}\label{eq: canonical lift to the theta group}
\bar{\sigma}_\nu   (\vartheta_\mu^{(k)}):=\vartheta_{\mu+\nu}^{(k)} \textrm{ and }   \chi(\vartheta_\mu^{(k)}):=\chi(\mu)\vartheta_\mu^{(k)}, 
\end{equation}
respectively,
for each
 $\mu,\nu\in \coker k\Phi$ and $\chi\in \Hom (\coker k\Phi, \C_K^\times)$.
In addition, by using these lifts, we may define the associated finite Heisenberg group as
\begin{equation}
    \label{eq: finite Heisenberg group at level k}
     H_k:=\langle (\coker k\Phi)^\sim, (\Hom (\coker k\Phi, \C_K^\times))^\sim\rangle \subset \calG(kL_\CK).
\end{equation}
Then
$H^0(X_{\C_K},kL_{\C_K})$ is an irreducible representation of $H_k$.

Given a lift $m_\mu$ of
$\mu\in \coker k\Phi$, the function
$\vartheta_\mu^{(k)}$ can be written more explicitly as
\begin{equation}\label{eq: normalizing scale}
    \vartheta_{\mu}^{(k)}=
\sum_{\lambda\in \Lambda} a_{k\Phi}^k (m_\mu + k\Phi (\lambda))z^{m_\mu + k\Phi (\lambda)}=
a_{k\Phi}^k(m_\mu)\sum_{\lambda\in \Lambda} a(\lambda)^kb(\lambda,m_\mu)z^{m_\mu+k\Phi(\lambda)}.
\end{equation}  
\begin{Def} \label{df: theta at level k with characteristic}
    We define 
    the \emph{theta function at level $k$ with characteristic $\mu\in \coker k\Phi$} by
    $$\theta^{(k)}_{m_\mu}:=\sum_{\lambda\in \Lambda} a(\lambda)^kb(\lambda,m_\mu)z^{m_\mu+k\Phi(\lambda)}\in H^0(X,kL),$$
 and write 
it simply as
$\theta^{(k)}_\mu$ after fixing the lift $m_\mu\in M$ of $\mu$.
\end{Def}
By the same discussion as above, $\{\theta^{(k)}_\mu\}_{\mu\in \coker k\Phi}$ forms a basis of $H^0(X,kL)$.
Since $kL$ is very ample for  $k\geq 3$,
we
 obtain
the following non-Archimedean Fubini--Study metric $\varphi_k$ on $L_{\C_K}^\na$:
$$
     \varphi_k:=\frac{1}{k}\max_{\mu\in \coker k \Phi} \log |\vartheta_\mu^{(k)}|. 
$$
In particular, by \eqref{eq: normalizing scale}, the metric $\varphi_k$ is 
obtained from
a non-Archimedean Fubini--Study metric 
$$
     \phi_k:=\frac{1}{k}\max_{\mu\in \coker k \Phi} \left(\log |\theta_\mu^{(k)}| +\log |a_{k\Phi}^k(m_\mu)|_\CK \right)
$$
on $L^\na$ as $(\phi_k)_{\C_K}=\varphi_k\in \mathrm{FS}(L^\na)_{\C_K}\subset \FS(L^\na_{\C_K}).$

\begin{Prop}\label{pr: scale relation}
For any $m\in M$, it holds that 
\begin{equation}\label{eq: scale relation}
-\log |a_{k\Phi}^k(m)|_\CK=kA_\Phi \left(\frac{1}{k}m\right).    
\end{equation}    
\end{Prop}
\begin{proof}
    By construction, we have
\begin{equation}\label{eq: scale relation for Lambda}
    a_{k\Phi}^k(k\Phi(\lambda))=a(\lambda)^k=a_\Phi(\Phi(\lambda))^k. 
\end{equation}
Since 
$a_{\Phi}(m)^2=b_{\Phi}(m,m)$ and
$a_{k\Phi}^k(m)^2=b_{k\Phi}(m,m)$ for any $m\in M$,
the bimultiplicativity of $b_\Phi$ and $b_{k\Phi}$ extends \eqref{eq: scale relation for Lambda} to \eqref{eq: scale relation}
 canonically after taking valuation.
\end{proof}
\begin{Rem}
    Since the extensions $b_\Phi$ and $b_{k\Phi}$ of $b$ are not canonical, we cannot expect 
    $a_{k\Phi}^k(km)=a_\Phi(m)^k$. However, $B_\Phi$ and $B_{k\Phi}$ are canonical. Hence, we obtain \pref{pr: scale relation}.
\end{Rem}

\begin{Thm}\label{th: theta functions and NA balanced metrics}
Keep the same notation as above. The polarized totally degenerate abelian variety
$(X,L)$ is asymptotically residually Chow polystable. 
Furthermore, for $k\geq 3$, 
the NAFS metric 
\begin{equation}
    \label{eq: NA balanced metric for pav}
\phi_k=\frac{1}{k}\max_{\mu\in \coker k \Phi} \left(\log |\theta_\mu^{(k)}| -kA_\Phi \left(\frac{1}{k}m\right) \right)\in \FS_k(L^\na)    
\end{equation}
is a NA balanced metric at level $k$.
\end{Thm}
\begin{proof}
 By \cite[\S~17]{Mum70}, $kL$ is very ample for $k\geq 3$.
As in \eqref{eq: normalizing scale}, for any such $k$, all normalized theta functions $\{\vartheta^{(k)}_\mu\}_{\mu\in \coker k\Phi}$ at level $k$ are defined after some finite base change $K'$ of $K$.
Denote by $\nm\in \calN^\str(V_k)$ a strictly Cartesian norm on $V_k:=H^0(X_\CK,kL_\CK)$
induced by the basis
$\{\vartheta^{(k)}_\mu\}_{\mu\in \coker k\Phi}$ of $V_k$.
Note that $$\varphi_k=(\phi_k)_\CK=\FS_k(\nm).$$
Then,
by \pref{th: Maculan's thm}, it suffices to show that 
the residual Chow point $[\widetilde{R_{X_\CK}}]\in \PP \left(\widetilde{W}_{\nm_\mathrm{Ch}}\right)$ induced by the norm $\nm\in \calN^\str(V_k)$ is GIT polystable.

 Consider the associated finite Heisenberg group 
 $H_k\subset \calG(kL_\CK)$ 
 defined in \eqref{eq: finite Heisenberg group at level k}.
 Fixing 
an isomorphism $V_k\simeq \C_K^{N_k}$ induced by
our basis
$\{\vartheta^{(k)}_\mu\}_{\mu\in \coker k\Phi}$ of $V_k$,
where $N_k:=\dim V_k$, it follows from \eqref{eq: canonical lift to the theta group} that
the finite Heisenberg group $H_k$ can be seen as
$$H_k \subset \GL_{N_k}(\C_K^\circ)\subset \GL_{N_k}(\C_K).$$
Following the same strategy as in \cite[Example 1]{Kem78},
we consider the image $\overline{H_k}\subset \PGL_{N_k}(\C_K)$ of $H_k$
and the inverse image $\scrH_k\subset \SL_{N_k}(\C_K)$ of $\overline{H_k}$
through the map 
$$\SL(V_k)(\CK)\surj \mathrm{PSL}(V_k)(\CK)\simeq\mathrm{PGL}(V_k)(\CK).$$
Then, since $\det M\in \C^\times$
for any $M\in H_k$,  
we have  
\begin{equation}
    \label{eq: scrH_k in SL}
    \scrH_k\subset \SL_{N_k}(\C_K^\circ).
\end{equation}
Furthermore, $V_k$ is an irreducible representation of $\scrH_k$.
Here,
as in \pref{ss: NA balanced metrics}, we regard 
the embedding $X\inj \PP(V_k^\vee)$ as a $d$-dimensional cycle of degree $\delta$ in $\PP(V_k^\vee)$, and set $W_k:=\left(\mathrm{Sym}^{\delta} (V_k^\vee)\right)^{\otimes (d+1)}$.
Then,
\eqref{eq: scrH_k in SL} implies that
$\scrH_k$ also acts on
$\PP\left(V_{k,\nm}^\circ\right)$ and
$\PP \left(W_{k,\nm_\mathrm{Ch}}^\circ\right)$, inducing actions of $$\widetilde{\scrH_k}:=\scrH_k \otimes \widetilde{\CK}$$ on 
$\PP \left(\widetilde{V}_{k,\nm}\right)$ and
$\PP \left(\widetilde{W}_{k,\nm_\mathrm{Ch}}\right).$
By construction, the corresponding residual Chow point
$[\widetilde{R_{X_\CK}}]\in \PP \left(\widetilde{W}_{k, \nm_\mathrm{Ch}}\right)$ is a fixed point under $\widetilde{\scrH_k}$. Since $\widetilde{V}_{k,\nm}$ is also an irreducible representation of $\widetilde{\scrH_k}$, it follows from \cite[Corollary~5.1]{Kem78} 
 that $[\widetilde{R_{X_\CK}}]\in \PP \left(\widetilde{W}_{\nm_\mathrm{Ch}}\right)$ is GIT polystable. In particular,
by \pref{th: Maculan's thm}, we conclude that $\phi_k$ is an NA balanced metric at level $k$.
\end{proof}

\begin{Cor}
    Keep the same notation as above. For $k\geq 3$, the NAFS metric
    $$ \varphi_k=\frac{1}{k}\max_{\mu\in \coker k \Phi} \log |\vartheta_\mu^{(k)}|\in \FS_k^\str(L_\CK^\na)$$
    is a critical metric on $L_\CK^\na$  at level $k$.
\end{Cor}
\begin{example}\label{eg: Aut of the special fiber may have positive dimension}
Unfortunately, a polarized totally degenerate abelian variety $(X,L)$ over $K$ may not be asymptotically residually Chow stable.
Indeed, suppose that $(X,L)$ is a principally polarized totally degenerate elliptic curve. We show that $(X,3L)$ is not residually Chow stable, although it is residually Chow polystable, as proved in \pref{th: theta functions and NA balanced metrics}. 

Since $L$ is a principal polarization, 
$3L$ corresponds to a homomorphism $\Phi: \Z\to \Z$ defined by $1 \mapsto 3$. In particular,  we have $\coker \Phi \simeq \Z/3\Z =\{0,1,2\}$.
Thus, 
the normalized theta functions
$\vartheta_\mu \in V:=H^0(X_\CK, 3L_\CK)$
with characteristic $\mu\in \{0,1,2\}$
form a basis of $V$. Since $3L$
is very ample,  we obtain a closed immersion
$X\inj \PP^2_\CK$ given by
$$p\longmapsto [\vartheta_0(p):\vartheta_1(p):\vartheta_2(p)].$$
As in \cite[\S~5, Case b]{Mum66}, the image of $X$ is defined by the equation
$$
    f:=x^3+y^3+z^3-3\lambda xyz=0
$$
in homogeneous coordinates $[x:y:z]$ on $\PP^2_\CK$,
where $\lambda\in \CK$ satisfies $\lambda^3\neq 1$.
 Moreover, if $|\lambda|_\CK < 1$, then $X$ is not totally degenerate. Therefore,  $|\lambda|_\CK \geq 1$ holds.

If $|\lambda|_\CK > 1$, then the reduction of the corresponding model $(\scrX, \scrL)$ over $\widetilde{\CK}$ is
$$\scrX_0=\left\{\widetilde{f}=xyz= 0\right\} \subset \PP^2_{\widetilde{\CK}}\simeq \PP^2_\C$$
in homogeneous coordinates $[x:y:z]$ on $\PP^2_\C$. By definition, the diagonal subgroup
$$(\C^\times)^2\simeq \left\{\diag (s,t,(st)^{-1}) \  \middle| \  s,t\in \C^\times \right\}\subset  \SL_3(\C)$$
preserves $\scrX_0\subset \PP^2_\C$. Thus, $\Aut (\scrX_0,\scrL_0)$ is not finite.

If $|\lambda|_\CK = 1$, then the reduction of the corresponding model $(\scrX, \scrL)$ over $\widetilde{\CK}$ is
$$\scrX_0=\left\{ \widetilde{f}=x^3+y^3+z^3-3\widetilde{\lambda} xyz= 0\right\} \subset \PP^2_{\widetilde{\CK}}\simeq \PP^2_\C$$
in homogeneous coordinates $[x:y:z]$ on $\PP^2_\C$. By the Jacobian criterion, $\scrX_0$ is smooth if and only if $\widetilde{\lambda}^3\neq 1$. Since $X$ is totally degenerate, we have $\widetilde{\lambda}^3 = 1$. Then
$$\widetilde{f}=x^3+y^3+z^3-3\widetilde{\lambda} xyz
=\left(x+y+\widetilde{\lambda}z\right)\left(x+\zeta_3y+\zeta^2_3\widetilde{\lambda}z\right)\left(x+\zeta_3^2y+\zeta_3\widetilde{\lambda}z\right).$$
Here, after a suitable coordinate change, we may assume that $\widetilde{f}=xyz$.
In particular,
by the same argument as above, it follows that
$\Aut (\scrX_0,\scrL_0)$ is not finite.

Thus, we conclude that the residual Chow point induced by the normalized theta functions is not stable.
However, as in the proof of \pref{th: theta functions and NA balanced metrics}, this residual Chow point is polystable.
If $(X,3L)$ were residually Chow stable, then \cite[Theorem~1]{GO26} would imply that every polystable residual Chow point is stable, which contradicts the above observation.
Hence, $(X,3L)$ is not residually Chow stable.
\end{example}

\subsection{Convergence of non-Archimedean balanced metrics}
Keep the notation from \pref{ss: theta functions and symmetries}. In this subsection, we establish the convergence of non-Archimedean balanced metrics for totally degenerate abelian varieties.
To see this, we consider the pullback of metrics on $L^\na$ along $p:T_K^\na\to X^\na$.
Since $p^*L^\na$ is trivial, given a continuous metric $\phi$ on $L^\na$, the pullback $p^*\phi$ is identified with a function $(p^*\phi)_1=-\log ||1||_{p^*\phi}$ on $T_K^\na$.
In particular, the non-Archimedean balanced metric $\phi_k$ corresponds to
a function $(p^*\phi_k)_1:T_K^\na \to \R$ defined as
$$ (p^*\phi_k)_1(x):=
  \frac{1}{k}\max_{\mu\in \coker k \Phi} \left(\log |\theta_\mu^{(k)}(x)| +\log |a_{k\Phi}^k(m_\mu)|_\CK \right)=\frac{1}{k}\max_{\mu\in \coker k\Phi} \left( \log |\vartheta_\mu^{(k)} (x)|\right). $$

The following is the key lemma to prove our main theorem \pref{th: convergence of NA balanced metrics}:

\begin{Lem}\label{lm: discrete Legendre dual} Set $k>1$. For 
 any
$n\in N_\R$ and
any $x\in \trop_K^{-1}(-n)\in T_{\C_K}^\na$,  we have
    \begin{equation} \label{eq: discrete Legendre dual}
        \frac{1}{k}\max_{\mu\in \coker k\Phi} \left( \log |\vartheta_\mu^{(k)} (x)|\right)=\max_{m\in \frac{1}{k} \Img \Phi} \left( \langle m,n\rangle - A_\Phi(m)\right).
   \end{equation}
\end{Lem}
\begin{proof}
It follows from \eqref{eq: normalizing scale} that
  \begin{align*}
      \log |\vartheta_\mu^{(k)} (x)| &\leq
     \max_{\lambda\in \Lambda} \left( \log |a_{k\Phi}^k (m_\mu + k\Phi (\lambda))|_\CK  e^{\langle m_\mu + k\Phi (\lambda), n \rangle} \right) \\
        &=\max_{\lambda\in \Lambda} \left(
        \langle m_\mu + k\Phi (\lambda), n \rangle - kA_\Phi\left(\frac{m_\mu}{k}+\Phi(\lambda)\right)\right)\\
        &=k\max_{m\in \frac{m_\mu}{k} + \Img \Phi} \left(
        \langle m, n \rangle - A_\Phi\left(m\right)\right).
  \end{align*}
   Hence  we obtain
   \begin{align*}
        \frac{1}{k}\max_{\mu\in \coker k\Phi} \left( \log |\vartheta_\mu^{(k)} (x)|\right)&\leq\max_{\mu\in \coker k\Phi}\max_{m\in \frac{m_\mu}{k} + \Img \Phi} \left(
        \langle m, n \rangle - A_\Phi\left(m\right)\right) \\
        &= \max_{m\in \frac{1}{k} \Img \Phi} \left( \langle m,n\rangle - A_\Phi(m)\right).
   \end{align*} 
Denote by $m_k\in  \frac{1}{k} \Img \Phi$ a maximizer of the  function $f_n(m):=\langle m, n \rangle - A_\Phi\left(m\right)$ over $\frac{1}{k} \Img \Phi$.
Then we have
$$m_k\in \frac{m_{\mu^*}}{k} + \Img \Phi$$ for some $\mu^*\in \coker k\Phi$.
If we have
\begin{equation}
    \label{eq: no cancellation}  \frac{1}{k}\log |\vartheta_{\mu^*}^{(k)} (x)| =\max_{m\in \frac{m_{\mu^*}}{k} + \Img \Phi} \left(
        \langle m, n \rangle - A_\Phi\left(m\right)\right),
\end{equation}
then the following computation
proves the assertion:
  \begin{align*}  
        \frac{1}{k}\max_{\mu\in \coker k\Phi} \left( \log |\vartheta_\mu^{(k)} (x)|\right)&\leq \max_{m\in \frac{1}{k} \Img \Phi} \left( \langle m,n\rangle - A_\Phi(m)\right) \\
        &=\langle m_k,n\rangle - A_\Phi(m_k) \\
        &=\max_{m\in \frac{m_{\mu^*}}{k} + \Img \Phi} \left(
        \langle m, n \rangle - A_\Phi\left(m\right)\right)\\
        &=\frac{1}{k}\log |\vartheta_{\mu^*}^{(k)} (x)| \\
        &\leq  \frac{1}{k}\max_{\mu\in \coker k\Phi} \left( \log |\vartheta_\mu^{(k)} (x)|\right).
   \end{align*} 
   Therefore, it suffices to show \eqref{eq: no cancellation}.
  Since the seminorm associated with $x$ is non-Archimedean, the equality in \eqref{eq: no cancellation} holds if 
  the function $f_n(m)$ attains its maximum over $\frac{m_{\mu^*}}{k} + \Img \Phi$ at a unique $m_k\in \frac{m_{\mu^*}}{k} + \Img \Phi$.
  Suppose otherwise. Namely, let $m_k^*$ be another maximizer of $f_n$ over $\frac{m_{\mu^*}}{k} + \Img \Phi$.
  Then $m_k^*$ is also a maximizer of $f_n$ over $\frac{1}{k}\Img \Phi$ with $f_n(m_k)=f_n(m_k^*).$
  However, since $f_n$ is strictly concave on $M_\R$, 
  we have $$f_n(m_t)>f_n(m_k)=f_n(m_k^*),$$
  where
  $m_t:=tm_k+(1-t)m_k^*\in M_\R$ for any $t\in (0,1)$.
  In particular, when $t\in \frac{1}{k}\Z$, where $k>1$, we have
  $$m_t=m_k^*+t(m_k-m_k^*)\in \frac{1}{k} \Img \Phi,$$
which contradicts the fact that $m_k$ is a  maximizer of $f_n$ over $\frac{1}{k}\Img \Phi$.
\end{proof}
\begin{Cor}
    The non-Archimedean balanced metric $\phi_k$ is a toric metric. Namely, 
    we have $$(p^*\phi_k)_1=(p^*\phi_k)_1\circ \trop_K.$$
\end{Cor}
Recall that $(\Lambda, a,b,\Phi)$ is the data associated to $L$.
As in 
\cite[Proposition~7.4.5]{BGGJK25} (see also
\cite[\S~4]{GS23}), such a toric metric $\phi$ on $L^\na$ is identified with a function 
$f:N_\R\to \R$ satisfying the cocycle rule, namely, 
\begin{equation}
    \label{eq: tropical cocycle condition}
     f(n+\trop_K(b^*(\lambda)))=f(n) -\log |a(\lambda)|_K+\langle \Phi(\lambda),n\rangle
\end{equation}
for any $n\in N_\R$ and any $\lambda\in \Lambda$. 
More precisely, 
such a function $f$ corresponds to
a toric metric $\phi_f$ characterized by
$$(p^*\phi_f)_1=f\circ \trop_K.$$
Conversely, 
a toric metric $\phi$ corresponds to
the function 
$$f_\phi:=(p^*\phi)_1\circ \frakG,$$
where $\frakG\colon N_\R\to T_K^\na$ denotes the Gauss section.

Since $A_\Phi: M_\R\to \R$ is a (strictly) convex function, we obtain its \emph{Legendre dual} $A_\Phi^\vee:N_\R\to \R$ as
\begin{equation}\label{eq: Legendre dual}
    A_\Phi^\vee(n):=  \sup_{m\in M_\R} \left( \langle m,n\rangle - A_\Phi(m)\right)= \max_{m\in M_\R} \left( \langle m,n\rangle - A_\Phi(m)\right).
\end{equation}
As its discrete analogue, we call the right hand side of \eqref{eq: discrete Legendre dual}, namely,
$$ \max_{m\in \frac{1}{k} \Img \Phi} \left( \langle m,n\rangle - A_\Phi(m)\right),$$
the \emph{discrete Legendre dual of $A_\Phi$ at level $k$}.

Since $A_\Phi: M_\R\to \R$ is the quadratic form such that $A_\Phi(m)=\frac{1}{2}B_\Phi(m,m)$ for $m\in M_\R$, 
we have 
$$A_\Phi(m)=\frac{1}{2} \langle m, [B_\Phi]m\rangle,$$
where 
$[B_\Phi]: M_\R\to N_\R$ is the $\R$-valued positive definite symmetric matrix corresponding  to the quadratic form $B_\Phi$ via the natural pairing $\langle -,-\rangle:M_\R\times N_\R\to \R$.
Fix $n\in N_\R$, and set $$f_n(m):=\langle m,n\rangle - A_\Phi(m).$$
Since 
\begin{equation}\label{eq: argmax}
    \nabla f_n=\nabla \left( \langle m,n\rangle -\frac{1}{2} \langle m, [B_\Phi]m\rangle\right)= n-[B_\Phi]m,
\end{equation} 
the Legendre dual $A_\Phi^\vee: N_\R\to \R$ is given by 
$$A_\Phi^\vee(n)= \langle [B_\Phi]^{-1}n,  n\rangle - \frac{1}{2}\langle  [B_\Phi]^{-1}n,[B_\Phi] [B_\Phi]^{-1}n\rangle=\frac{1}{2}\langle  [B_\Phi]^{-1}n,n\rangle$$
for any $n\in N_\R$. In particular, $A_\Phi^\vee(-n)=A_\Phi^\vee(n).$
Moreover, it follows  from direct computation that $A_\Phi^\vee$ satisfies \eqref{eq: tropical cocycle condition}, which yields a toric metric $\phi_{A_\Phi^\vee}$ on $L^\na$.
As in \cite[Theorem~4.3]{Liu11}, the toric metric $\phi_{A_\Phi^\vee}$ coincides with the NACY metric $\phi^\mathrm{NACY}$on $L^\na$.

\begin{Thm}
\label{th: convergence of NA balanced metrics}
For the polarized totally degenerate abelian variety $(X,L)$ over $K$, the non-Archimedean balanced metrics $\phi_k$ on $L^\na$ converge  to the NACY metric $\phi^\mathrm{NACY}$ on $L^\na$ in the $\mathscr{C}^0$-topology as $k\to \infty$ after suitable normalization.
Furthermore, we have
    $$ \sup_{x\in X^\na} \left| \phi^\mathrm{NACY}-\phi_k\right| = O(k^{-2}).$$
\end{Thm}
\begin{proof}
Since $\phi_k$ and $\phi^\mathrm{NACY}=\phi_{A_\Phi^\vee}$ are toric, 
by \pref{lm: discrete Legendre dual} and \eqref{eq: Legendre dual},
it suffices to show that 
  \begin{equation}\label{eq: key estimate}
      \sup_{n\in N_\R} \left| d_k(n)\right| = O(k^{-2}),
  \end{equation} 
where
 $$d_k(n):=\max_{m\in M_\R} \left( \langle m,n\rangle - A_\Phi(m)\right)-\max_{m\in \frac{1}{k} \Img \Phi} \left( \langle m,n\rangle - A_\Phi(m)\right).$$
By definition, we have $d_k\geq 0$.
 Recall that, for any $n\in N_\R$, the function $f_n=\langle m,n\rangle - A_\Phi(m)$
 attains its maximum over $M_\R$ at $m_0=[B_\Phi]^{-1}n$ by \eqref{eq: argmax}. Then we obtain
\begin{align*}
    d_k(n)
         &= \langle m_0,n\rangle -\frac{1}{2} \langle m_0, [B_\Phi]m_0\rangle - \max_{m\in \frac{1}{k} \Img \Phi} \left( \langle m,n\rangle - \frac{1}{2} \langle m, [B_\Phi]m\rangle \right)\\
         &=  \min_{m\in \frac{1}{k} \Img \Phi} \left( \langle m_0,n\rangle -\frac{1}{2} \langle m_0, [B_\Phi]m_0\rangle - \langle m,n\rangle + \frac{1}{2} \langle m, [B_\Phi]m\rangle \right)\\
         &=  \min_{m\in \frac{1}{k} \Img \Phi} \left( \langle m_0,[B_\Phi]m_0\rangle -\frac{1}{2} \langle m_0, [B_\Phi]m_0\rangle - \langle m, [B_\Phi]m_0\rangle + \frac{1}{2} \langle m, [B_\Phi]m\rangle \right)\\
         &=  \min_{m\in \frac{1}{k} \Img \Phi} \left(\frac{1}{2} \langle m_0, [B_\Phi]m_0\rangle - \langle m, [B_\Phi]m_0\rangle + \frac{1}{2} \langle m, [B_\Phi]m\rangle \right)\\
         &=  \min_{m\in \frac{1}{k} \Img \Phi} \frac{1}{2} \langle m_0-m, [B_\Phi](m_0-m)\rangle \\
         &=  \min_{m\in \frac{1}{k} \Img \Phi}  A_\Phi(m_0-m) = \frac{1}{k^2}  \min_{\lambda \in  \Img \Phi}  A_\Phi(km_0-\lambda)= \frac{1}{k^2}  \min_{\lambda \in  \Img \Phi}  A_\Phi(k[B_\Phi]^{-1}n-\lambda).\\
\end{align*}
Since an induced homomorphism $k[B_\Phi]^{-1}:N_\R\to M_\R$ is an isomorphism, we have
\begin{align*}
   \sup_{n\in N_\R} d_k(n)
         &= \frac{1}{k^2} \sup_{n\in N_\R} \min_{\lambda \in  \Img \Phi}  A_\Phi(k[B_\Phi]^{-1}n-\lambda)\\
          &= \frac{1}{k^2} \sup_{m\in M_\R} \min_{\lambda \in  \Img \Phi}  A_\Phi(m-\lambda).\\
\end{align*}
Here, consider a function $\delta :M_\R\to \R$ defined by
$$\delta(m):=\min_{\lambda \in  \Img \Phi}  A_\Phi(m-\lambda).$$ 
By definition, the function $\delta$
descends to a function on $M_\R/ \Img \Phi$.
Since $M_\R/ \Img \Phi$ is compact, the function $\delta$ is bounded.
Namely, there exists $C>0$ such that 
$$\sup_{m\in M_\R}\delta(m)=\sup_{m\in M_\R} \min_{\lambda \in  \Img \Phi}  A_\Phi(m-\lambda)\leq C.$$
Hence we obtain \eqref{eq: key estimate}.
\end{proof}

As a future direction, we consider the case of general Calabi--Yau varieties.

\begin{Prop}\label{pr: asymptotically Chow stable}
    Any polarized smooth Calabi--Yau variety $(X,L)$ over a field $F$ of characteristic zero is asymptotically Chow stable.
\end{Prop}
\begin{proof}
    As in \cite[Proposition~1.14]{GIT}, GIT stability is invariant under field extension.
Hence,
by the Lefschetz principle, we may assume that $F=\C$.
Then  $\Aut(X,L)$ is discrete. 
Indeed, since $L$ is ample, $G:=\Aut^0(X,L)$ is a connected linear algebraic group.
As $G$ acts faithfully on $X$, 
for a general point $x\in X$, the orbit $Gx$ of $x$ has positive dimension if $\dim G>0$.
In particular, $Gx$ contains a rational curve passing through $x$, which means that $X$ is uniruled.
However,
$X$ is not uniruled; see, for instance, \cite{BDPP13}.
Hence $\dim G=0$, which implies that $\Aut(X,L)$ is discrete. 
Then it follows from \cite[Theorem~3]{Don01} and \cite[Theorem~0.1]{Luo98} that 
$(X,L)$ is asymptotically Chow stable.
\end{proof}

In the proof of \pref{pr: asymptotically Chow stable}, we proved, in particular, that $\Aut(X,L)$ is discrete when $F=\C$. The author learned this fact in conversation with Yuji Odaka, who noted that it is well known to experts and explained how it can be proved. The argument presented here follows his explanation. However, the author takes full responsibility for any errors.

By \pref{pr: asymptotically Chow stable} and \pref{th: Zhang's semistable reduction}, the following holds:
\begin{Cor}
    For any 
    smooth polarized Calabi--Yau variety $(X,L)$ over $K$ and any sufficiently large $k\in \Z_{>0}$, 
    there exists a critical metric $\varphi_k$ on $L_\CK^\na$ at  level $k$.
\end{Cor}

By \pref{cr: asym res Chow st yields a sequence of NA balanced}, 
to establish the existence and uniqueness of a non-Archimedean balanced metric for every sufficiently large $k$, 
it is natural  to consider the following question:

\begin{Q.} \label{q: existence of NA balanced}
Let $(X,L)$ be an asymptotically Chow stable  smooth polarized variety over $K$.
When is $(X,L)$ asymptotically residually Chow stable?
Or, at least, when does $(X,L)$ admit NA balanced metrics at every sufficiently large level $k$?
\end{Q.}
Furthermore, the results of \cite{Don01} motivate the following question:

\begin{Q.} 
Let $(X,L)$ be an asymptotically Chow stable smooth polarized variety over $K$.
Suppose that $(X,L)$ admits an NA balanced metric $\phi_k$ at every sufficiently large level $k$.
When does the sequence $\{\phi_k\}_k$ converge?
Moreover, if the limit
$$
\phi_\infty:=\lim_{k\to\infty}\phi_k
$$
exists, does it admit an analytic characterization analogous to that of cscK metrics?
\end{Q.}

In particular, it is known that every maximally degenerate smooth polarized Calabi--Yau variety $(X,L)$ over $K$ admits 
a canonical metric on $L^\na$ called
the NACY metric $\phi^\mathrm{NACY}$.
Thus, specializing the above questions to the Calabi--Yau setting, we may ask the following:
\begin{Q.} \label{q: NA estimate for NACY}
Under what conditions does a maximally degenerate smooth polarized Calabi--Yau variety $(X,L)$ over $K$ admit NA balanced metrics $\phi_k$ at every sufficiently large level $k$? If such metrics exist, after suitable normalization, do they approximate the NACY metric $\phi^\mathrm{NACY}$ on $L^\na$? More precisely, can we establish the estimate
$$
\sup_{x\in X^\na}|\phi^\mathrm{NACY}-\phi_k|=O(k^{-1})?
$$
\end{Q.}

\begin{Rem}

As studied in \cite{Li25}, NACY metrics can also be considered beyond the maximally degenerate setting. However, as pointed out in \cite{BJ17}, their definition in this generality appears to require Archimedean data. Therefore, in this section, we restrict our attention to the maximally degenerate case in order to formulate the problem entirely within the non-Archimedean framework, whereas, when we consider hybrid framework as in \pref{sc:Applications via hybrid geometry}, we may consider such generalized situation; see \pref{rm: non maximally degenerate case}.

\end{Rem}

\section{Applications via hybrid geometry}
\label{sc:Applications via hybrid geometry}
Non-Archimedean estimates, such as \pref{th: convergence of NA balanced metrics}, often yield uniform estimates for a given family.
To give such an example,
in this section, we consider a holomorphic (degenerating) family $(\calX,\calL)$ of polarized
varieties over the punctured disc $\Delta^*:=\{t\in \C\ | \ 0< |t|_\infty<e^{-1}\}$,
where $\va_{\infty}$ is the usual absolute value on $\C$,
such that the structure map
$\pi:\calX\to \Delta^*$ is flat,
and the map $\pi$ can be extended to a proper flat map 
$\pi_\scrX:\scrX \to \Delta:=\{t\in \C \ | \ |t|_\infty<e^{-1}\}$ with $\scrX$ normal.
Let 
 $\nm_\hyb : \C \to \R_{\geq 0}$ be the \emph{hybrid norm on $\C$} given as
 $$\nm_{\rm hyb}:=\max\{ \va_{\infty},\va_{\rm triv}\},$$ where $\va_{\rm triv}$ is the trivial valuation on $\C$. 
Then $\calX$ is defined over the \emph{hybrid (circle) ring}
 \begin{equation}
     \label{eq: hybrid circle ring}
 \scrA:=\left\{f=\sum_{\alpha \in \Z} c_\alpha t^\alpha \in K \ \middle| \ c_\alpha\in \C, ||f||_{\hyb, e^{-1}}:=\sum_{\alpha\in \Z} \left(||c_\alpha||_{\rm hyb} \cdot e^{-\alpha} \right) <\infty \right\}.
 \end{equation}
 Here, we may apply the Cauchy--Hadamard theorem to \eqref{eq: hybrid circle ring}.
That is,
 after multiplying 
 $f\in \scrA$
 by a sufficiently large power of $t$,
 the radius of convergence  is at least $e^{-1}$. Thus, for any
 $f\in \scrA$ and  any $t\in \Delta^*$, the value $f(t)\in \C$ is well-defined.
 By abuse of notation, we still denote by $\pi:\calX\to \scrA$ the structure morphism induced by $\pi_\scrX$.
 In addition, the line bundle $\calL$ is assumed to be $\pi$-ample.
 Then, since $\scrA \subset K$, we obtain the associated polarized variety $(X,L):=(\calX_K,\calL_K)$ over $K$.
 Furthermore,
since $\calL$ is $\pi$-ample, by the Serre vanishing theorem, for all sufficiently large $k$,
it holds that
$R^i\pi_* (k \calL)=0$ for any $i>0$.
Since $\pi$ is flat, for any such $k$, the dimension $h^0(\calX_t,k\calL_t)=\chi (\calX_t,k\calL_t)$
is constant in $t\in \Delta^*$, which is denoted by $N_k$.

\subsection{Hybrid analytic spaces and hybrid continuous metrics}
\label{ss: hybrid spaces}
The framework of Berkovich geometry is not limited to the geometry over non-Archimedean fields considered so far, but naturally extends to the geometry over more general Banach rings such as so-called \emph{geometric base rings}; see \cite{LP24} for details.
Based on the framework, we consider the geometry over $\scrA$.
Here, note that $\scrA$ is a Banach ring with respect to the norm $||\cdot ||_{\hyb, e^{-1}}$ as in \eqref{eq: hybrid circle ring}. Furthermore, it follows from \cite[Proposition~2.1.1]{Poi10} that its Berkovich spectrum $\scrM(\scrA)$ is isomorphic to the closed disc $\overline{\Delta}:=\{t\in \C\ | \ |t|_\infty\leq e^{-1}\}$ of radius $e^{-1}$.
Since we assume $\calX$ is defined over $\scrA$, we can obtain its analytification
$\calX^\hyb$, which is called the \emph{hybrid analytification}.
In particular, $\calX^\hyb$ is endowed with a topology and a structure sheaf. Furthermore, the line bundle $\calL$ on $\calX$ induces a line bundle $\calL^\hyb$ on $\calX^\hyb$.
A notable feature of $\calX^\hyb$ is that the structure morphism $\pi:\calX\to \scrA$ induces a fibration
$$
    \pi^\hyb:\calX^\mathrm{hyb}\to \Delta=\Delta^*\cup \{ 0\}
$$
such that there exist isomorphisms $h:\calX^\mathrm{hyb}|_{\Delta^*}\simeq \calX|_{\Delta^*}$ and $\calX^\hyb|_{t=0}\simeq X^\na$. In particular, for any $t\in \Delta^* \left(\subset \overline{\Delta}\simeq \scrM(\scrA)\right)$, 
an induced isomorphism $$h_t:\calX^\hyb_t:=\calX_{\scrH(t)}^\na\simeq \calX_t$$ 
can be regarded as a rescaling by $c(t):=\frac{-1}{\log |t|_\infty}$
since the completed residue field $\scrH(t)$ at $t \in \Delta^*$ is given by the field $\C$ with a rescaled absolute value $|\cdot(t)|:=|\cdot|_\infty^{c(t)}$.
Hence, $h_t:\calX^\hyb_t\simeq \calX_t$ 
induces $$\calL_t^\hyb\simeq c(t)\calL_t$$ as an $\R$-divisor on $\calX_t$. 
Note that $c(t)\to0$ as $t\to 0$.
In this sense, the geometry over $\scrA$ is useful for studying (degenerating) families of polarized varieties.
This approach was initiated by \cite{BJ17}, and further developed by \cite{Fav20} and \cite{PS23} with respect to metrics.

\begin{Def}\label{df: hybrid continuous metric}
    A \emph{(hybrid) continuous metric $\phi$} on $\calL^\hyb$ is a family of norms
$$ \nm_{\phi(x)}: \calL(x):=\calL\otimes \scrH(x)\to \R_{\geq 0},\ \mathrm{for \ each\ } x\in \calX^\hyb$$
 such that for any local trivializing section $s$ of $\calL$ on an open subset $U\subset \calX$, the induced function 
$$\phi_s:=-\log ||s||_\phi :U^\hyb\to \R $$
is continuous. 

\end{Def}
Since \pref{df: hybrid continuous metric}  is obtained from \pref{df: continuous metric on nA spaces} merely by replacing $(X,L)$ with $(\calX,\calL)$,
the same argument as at the beginning of \pref{ss: NA pluripotential theory} applies verbatim and gives the following:
\begin{itemize}
    \item For any two continuous metrics $\phi$ and $\psi$ on $\calL^\hyb$, the difference $\phi-\psi$ 
    is a continuous metric on the trivial sheaf $\calO_\calX^\hyb$, 
    and it is identified with a continuous function on $\calX^\hyb$ by $-\log||1||_{\phi-\psi}: \calX^\hyb \to \R$.
    \item Since $\calX$ is proper over $\scrA$, 
    $\calX^\hyb$ is compact. Hence, for any two continuous metrics $\phi$ and $\psi$ on $\calL^\hyb$,
the \emph{uniform norm} 
\begin{equation}\label{eq: hybrid uniform norm}
\sup_{ \calX^\hyb} |\phi-\psi| :=   \sup_{ \calX^\hyb} \left|-\log||1||_{\phi-\psi}\right|
\end{equation}
 is well-defined.
\end{itemize}
Furthermore,
let $\{s_i\}_{i\in I}\subset H^0(\calX,k\calL)$ be a finite set of  sections for some $k\geq 1$ with no common zeros, and let 
$\{\lambda_i\}$ be a finite set of real numbers indexed by the same set $I$.
Then
\begin{equation}\label{eq: hybrid FS}
    \left(\frac{1}{k} \max_{i\in I}\left\{\log|s_i| +\lambda_i\right\}\right)_s(x):=\frac{1}{k} \max_{i\in I}\left\{-\log|(s^k/s_i)(x)| +\lambda_i\right\}
\end{equation}
for any local trivializing section $s$ of $\calL$, 
defines a hybrid continuous metric
 in a similar way to \pref{eg: log|s|}.

\begin{Def}
A hybrid continuous metric on $\calL^\hyb$ of the form \eqref{eq: hybrid FS}
 is called a \emph{(hybrid) tropical Fubini--Study metric on $\calL^\hyb$}.
Denote by $\FS^\tau(\calL^\hyb)$ the set of hybrid tropical Fubini--Study metrics on $\calL^\hyb$.
A hybrid continuous metric on $\calL^\hyb$ is said to be \emph{continuous plurisubharmonic} (\emph{cpsh} for short) if 
$\phi$ can be written as the uniform limit of a decreasing net $\{\phi_i\}_i\subset \FS^\tau(\calL^\hyb)$. Denote by $\CPSH (\calL^\hyb)$ the set of (hybrid) cpsh metrics on $\calL^\hyb$, endowed with the topology of uniform convergence on $\calX^\hyb$ given by \eqref{eq: hybrid uniform norm}.
\end{Def}
For any $t\in \Delta$ and any hybrid cpsh metric $\phi\in \CPSH(\calL^\hyb)$, we write, by abuse of notation, $\phi(t)$ for its restriction to $\calX^\hyb_t$. Then
    it follows from \pref{df: NAFS and cpsh} and \cite[Theorem~2.11]{PS23} that
    \begin{equation}\label{eq: hyb cpsh metric}
        \phi (t) \in 
  \begin{cases*}
   \CPSH(\calL_t^\hyb)=\CPSH(c(t)\calL_t) & if $t\neq 0$, \\
   \CPSH(\calL_0^\hyb)=\CPSH(L^\na)     & if $t=0$.
  \end{cases*} 
    \end{equation} 
In particular, $c(t)^{-1}\phi(t)$ may be regarded as a cpsh metric on $\calL_t$. By abuse of notation, we sometimes denote this metric on $\calL_t$ by $\phi_t$ so as to keep the notation compatible with the scaling on the complex side.
In view of \eqref{eq: hyb cpsh metric}, every hybrid cpsh metric on $\calL^\hyb$ induces a unique cpsh metric on $L^\na$. However, the converse is false. That is, if a cpsh metric on $L^\na$ admits an extension to a hybrid cpsh metric on $\calL^\hyb$, then such an extension is not unique in general; see \pref{eg: extendability of NAFS}. 
On the other hand, 
if two hybrid cpsh metrics $\phi$ and $\psi$ on $\calL^\hyb$ satisfy $\phi|_{\Delta^*}=\psi|_{\Delta^*}$, then $\phi=\psi$.
\begin{example} \label{eg: extendability of NAFS}
   Consider a NAFS metric $\phi:=\phi^\mathrm{NA}$ on $L^\na$ given by  a finite set $\{s_i\}\subset H^0(X,kL)$ with no common zeros for some $k\in\Z_{>0}$ and  $\lambda_i\in \R$ for each $i\in I$.
   Suppose that these $\{s_i\}\subset H^0(X,kL)$ can be lifted to sections of $H^0(\calX,k\calL)$.
   Then, by possibly shrinking the radius of $\Delta$, we may assume that $\{s_i\}\subset H^0(\calX,k\calL)$ still has no common zeros. For those $\{s_i\}\subset H^0(\calX,k\calL)$ and  $\{\lambda_i\}$, the associated hybrid tropical Fubini--Study metric $\phi^\hyb_\infty$ is a cpsh metric on $\calL^\hyb$ such that its restriction $\phi^\hyb_\infty (0)$ to $\calX^\hyb_0\simeq X^\na$ is the given $\phi$. In other words, $\phi$ can be extended to the hybrid cpsh metric  $\phi_\infty^\hyb$  on $\calL^\hyb$.
   Note that this extension is not  unique. Indeed, by the same argument as in \cite[Lemma~3.6]{PS23}, for any $p>0$, there 
exists a hybrid cpsh metric
    $\phi_p^\hyb$ on  $\calL^\hyb$ such that 
\begin{equation*}
        \phi_p^\hyb (t) =
  \begin{cases*}
   \frac{1}{pk} \log\left(\sum_{i\in I} e^{p\lambda_i} |s_i|^p\right)= \frac{1}{pk} \log\left(\sum_{i\in I} |e^{\lambda_i} s_i|^p\right)& if $t\neq 0$, \\
   \frac{1}{k} \max_{i\in I}\left\{\log|s_i| +\lambda_i\right\} =\phi & if $t=0$.
  \end{cases*} 
    \end{equation*} 
    We call it the \emph{hybrid $L^p$-metric associated to $\phi$.} 
In particular, 
the hybrid $L^2$-metric is called the \emph{hybrid Bergman metric} in \cite{PS23}.
Furthermore, by allowing $p=\infty$, we may regard hybrid tropical Fubini--Study metrics as \emph{hybrid $L^\infty$-metrics}.
When  $\{s_i\}\subset H^0(\calX,k\calL)$ forms a basis of $H^0(\calX_t,k\calL_t)$ for some $t\in \Delta^*$, the restriction   of the hybrid Bergman metric
\begin{equation*}
\phi_2^\hyb(t)=\frac{1}{2k} \log\left(\sum_{i\in I} |e^{\lambda_i} s_i|^2\right)    
\end{equation*}
corresponds to the Hermitian norm $\nm$ on $H^0(\calX_t,k\calL_t)$ diagonalized by $\{s_i\}$ with  $||s_i||=e^{-\lambda_i}$ for each $i$, which also corresponds to a usual complex Fubini--Study metric on $\calL_t$. 
\end{example}

By \eqref{eq: hyb cpsh metric},
considering hybrid cpsh metrics on $\calL^\hyb$ allows us to treat cpsh metrics on the given family $\calL$ and non-Archimedean cpsh metrics on $L^\na$ simultaneously, while also exploiting the topological properties of $\calX^\hyb$ as follows:

\begin{Lem}\label{lm: usc for sup}
    Let $Y$ be a locally compact Hausdorff topological space, and let $\pi \colon X \to Y$ be a proper map.  
Then, for any continuous function $f \in \mathscr{C}^0(X)$ on $X$, the function $u:Y\to \R$ defined as
$$
y \longmapsto u(y):=\sup_{x \in \pi^{-1}(y)} f(x)
$$
is upper semicontinuous on $Y$.
\end{Lem}
\begin{proof}
   Since $Y$ is locally compact and Hausdorff, the proper map $\pi$ is continuous.
   We may assume that $Y$ is compact, which particularly implies that $X$ is also compact. 
Fix $y \in Y$. We show that
$$
\limsup_{y'\to y} u(y') \leq u(y).
$$
Suppose, for contradiction, that $$\limsup_{y'\to y} u(y') > u(y).$$ Then there exists
$\varepsilon>0$ and a sequence $\{y_i\}_{i\in \Z_{>0}}$ converging to $y$ such that
$
u(y_i) > u(y) + 2\varepsilon$ for all $i$.
Since $u$ is defined as a supremum, for each $y_i$, we can choose
$x_i \in \pi^{-1}(y_i)$ such that
$$
f(x_i) + \varepsilon \geq u(y_i)>u(y)+2\varepsilon.
$$
Since $X$ is compact, the sequence  $\{x_i\}_{i\in \Z_{>0}}$ admits
a convergent subsequence. Let $x\in X$ be its limit. Since $\pi$ is continuous, we have $\pi(x)=y$. 
Furthermore, by
continuity of $f$, we also obtain
$$
f(x) \geq u(y) + \varepsilon,
$$
which contradicts the definition of $u(y)$. 
\end{proof}
Since $\pi^\hyb:\calX^\hyb\to \Delta$ is a continuous map from the compact space $\calX^\hyb$ to the Hausdorff space $\Delta$, the map $\pi^\hyb$ is proper. Hence \pref{lm: usc for sup} implies the following:

\begin{Prop}\label{pr: uniform estimate}
    Fix a hybrid cpsh metric $\phi$ on $\calL^\hyb$.
Suppose that there exists a sequence $\{\phi_k\}_{k\in \Z_{>0}}$ of hybrid cpsh metrics on $\calL^\hyb$ such that
the non-Archimedean estimate
$$
\sup_{X^\na} |\phi(0)-\phi_k(0)| = O(k^{-1})
$$ holds
for all sufficiently large $k$. Then, after possibly shrinking the radius of $\Delta$, the uniform estimate
$$
\sup_{\calX^\hyb} |\phi-\phi_k| = O(k^{-1})
$$
holds
for any such $k$.

\end{Prop}

\begin{Rem}
The uniform estimate holds on the shrunken disc $\Delta_k$.
Denote by $r_k$ the radius of $\Delta_k$.
To conclude that the hybrid cpsh metrics $\phi_k$ converge to $\phi$ in the $\mathscr C^0$-topology as $k\to\infty$, we need 
\begin{equation}\label{eq: for hybrid convergence}
\liminf_{k} r_k>0.    
\end{equation}
However, it remains unclear whether this holds.
\end{Rem}

Applying \cite[Theorem~1.2]{Li25} to a maximally degenerating family $(\calX,\calL)$ of polarized Calabi--Yau varieties,
the family $\phi=\{\phi_t^\mathrm{CY}\}_{t\in \Delta^*}$ of Calabi--Yau metrics 
on $\calL=\{\calL_t\}_{t\in \Delta^*}$ extends to a hybrid cpsh metric $\phi^\mathrm{hyb}$
whose restriction to $X^\na$ is the NACY metric on $L^\na$. 
Hence \pref{pr: uniform estimate} implies the following:
\begin{Cor}\label{cr: unif estimate for CY metrics}
    If \pref{q: NA estimate for NACY} has an affirmative answer and the basis of $H^0(X,kL)$ used to describe the NA balanced metric $\phi_k$ can be lifted to sections of $H^0(\calX,k\calL)$, then the uniform estimate
    $$ \sup_{\calX_t} c(t)|\phi_t^\mathrm{CY} - c(t)^{-1} (\phi_k)_p^\hyb (t)|= O(k^{-1})$$
    holds
    for any $p\in(0,\infty]$ and all sufficiently small $t\in \Delta^*$, where $(\phi_k)_p^\hyb$ is the hybrid $L^p$-metric associated to the NA balanced metric $\phi_k$.
\end{Cor}

\begin{Rem}\label{rm: another approximation of CY}
Recall that $\FS_k^\str (L^\na_\CK)$
carries a natural $\SL_{N_k}(\C_K)$-action, where $N_k=h^0(X,kL)$.
Then the action of $\SL_{N_k}(K)\subset \SL_{N_k}(\C_K)$ descends to an action on $\FS_k^\str (L^\na_\CK)|_L$.
Indeed, for any $g\in \SL_{N_k}( K)$ and any diagonalizable norm
$\nm\in \calN^\diag (V_k)$, where  $V_k:=H^0(X,kL)$, the induced norm $\nm_g$, 
given by $$||v||_g:=||gv||$$ for $v\in V_k$,
is again  diagonalizable.
In fact, if $\|\cdot\|$ is diagonalized
by a $K$-basis $\{s_i\}$ of $V_k$, then $\|\cdot\|_g$ is diagonalized
by the $K$-basis $\{g^{-1}s_i\}$.
Here,  the NA balanced metric $\phi_k$ can be written as
$$ \frac{1}{k}\max_i \{\log |s_i|+\lambda_i\}$$
for some basis $\{s_i\}$ of $V_k$ and some real numbers $\lambda_i\in \R$. 
By definition, the NA balanced metric $\phi_k$ is invariant under the action of $\SL_{N_k}( \C_K^\circ)$.
In particular, it is invariant under the action of $\SL_{N_k}( K^\circ \cap \scrA)$.
Here, by the definition of $\scrA$, we may evaluate $g\in \SL_{N_k}(K^\circ \cap \scrA)$ at
$t\in \Delta^*$, thereby obtaining $g(t)\in \SL(N_k,\C)$.
Since this $\SL_{N_k}( K^\circ \cap \scrA)$-action  does not affect  the assumptions of \pref{cr: unif estimate for CY metrics}, for 
any $g\in \SL_{N_k}( K^\circ \cap \scrA)$, it follows from \pref{cr: unif estimate for CY metrics} that 
 $$ \sup_{\calX_t}c(t) |\phi_t^\mathrm{CY} - c(t)^{-1} (g^*\phi_k)_p^\hyb (t)|= O(k^{-1}).$$
When $p=2$, the metric $c(t)^{-1}(g^*\phi_k)_2^\hyb(t)$ yields a usual complex Fubini--Study metric on $\calL_t$ for each $t\in \Delta^*$. More explicitly, this metric corresponds to a Hermitian norm $\nm$ on $H^0(\calX_t,k\calL_t)$ diagonalized by $\{g^{-1}(t)s_i(t)\}$ with $||g^{-1}(t)s_i(t)||=e^{-\lambda_i}$ for each $i$. 
However, since $g(t)\in \SL_{N_k}( \C) \nsubseteq U(N_k)$, where $U(N_k)$ is the unitary group,
the action $g(t)$ may change $c(t)^{-1} (\phi_k)_2^\hyb (t)$.
In particular, the condition
\eqref{eq: for hybrid convergence} for the sequence $ \{(\phi_k)_2^\hyb\}$ does not imply
the one for $\{(g^*\phi_k)_2^\hyb\}$.
However,
if \eqref{eq: for hybrid convergence} holds for both sequences $ \{(\phi_k)_2^\hyb\}$ and  
$\{(g^*\phi_k)_2^\hyb\}$, then, in particular, 
this action  yields another sequence of complex Fubini--Study metrics converging to the Calabi--Yau metric $\phi_t^\mathrm{CY}$ in the $\mathscr{C}^0$-topology for sufficiently small $t\in \Delta^*$. 
\end{Rem}

\subsection{Applications to totally degenerate abelian varieties}
In this subsection, 
we discuss the conditions imposed so far, specialized to the case of a
totally degenerating family $(\calX,\calL)$ of polarized abelian varieties.
That is, we consider a family 
$(\calX,\calL)$ of polarized abelian varieties over $\Delta^*$ or $\scrA$ such that the base change
$(X,L):=(\calX_K,\calL_K)$ is a polarized totally degenerate abelian variety.
Let $(\Lambda,a,b,\Phi)$ be the data associated to $L:=\calL_K$ as in \eqref{eq: data for L}.
By \cite[Lemma~2.6]{GO24}, the hybrid analytification $(\calX^\hyb,\calL^\hyb)$ admits a uniformization that extends the Mumford uniformization  \eqref{eq: Mumford--Tate uniformization}. In particular, we may assume that the map
$b:\Lambda \times M\to K^\times$  takes values in
$\scrA^\times$ and
 obtain $$\calX\simeq T_\C(\C) \times \Delta^* /\beta(\Lambda),$$
 where $T_\C=\Spec \C[M]$. 
Furthermore, we may  assume that 
the map $a:\Lambda\to K^\times$  also takes values in $\scrA^\times$.
Hence the classical theta function 
$$\theta^{(k)}(t)=\sum_{\lambda\in \Lambda} a(\lambda)^k(t) z^{k\Phi(\lambda)}$$ is well-defined for each $t\in \Delta^*$; see, for instance, \cite[\S~6]{GH78}.
Similarly, $\theta_{m_\mu}^{(k)}(t)$ is also well-defined.
This means that the basis of $H^0(X,kL)$ used to describe the NA balanced metric $\phi_k$ can be lifted to sections of $H^0(\calX,k\calL)$.
Hence, by \pref{lm: usc for sup} and \pref{cr: unif estimate for CY metrics}, we obtain the following:
\begin{Thm}
\label{th: approximation of CY by L^p-metrics}
For $k\geq 3$, after possibly shrinking the radius of $\Delta$,
     we obtain the uniform estimate
    $$ \sup_{\calX_t} c(t)|\phi_t^\mathrm{CY} - c(t)^{-1} (\phi_k)_p^\hyb (t)|= O(k^{-2})$$
    for any $p\in (0,\infty]$ and all sufficiently small $t\in \Delta^*$, where $(\phi_k)_p^\hyb$ is the hybrid $L^p$-metric associated to the NA balanced metric $\phi_k$.
\end{Thm}

As observed in \eqref{eq: normalizing scale}, to define the normalized theta functions $\vartheta_\mu^{(k)}$ with characteristics, we need, in addition to $\theta_{m_\mu}^{(k)}$, the normalizing factor
$a_{k\Phi}^k(m_\mu)\in \overline{K}.$
Thus, for a fixed $k$, 
after making a finite base change $K'/K$ such that $a_{k\Phi}^k(m_\mu)\in K'$ for all $\mu\in \coker k\Phi$, we may assume that
$\vartheta_\mu^{(k)}(t)$ is well defined for $t\in \Delta^*$.
By \cite[Theorem~2.1]{WZ17}, 
such sections $\left\{\vartheta_\mu^{(k)}(t)\right\}_{\mu\in \coker k\Phi}$ of $H^0(\calX,k\calL)$
yield
a family of complex balanced metrics on $\calL$ at level $k$ over $\Delta^*$.
In particular,
for any $t\in \Delta^*$, the restriction
$c(t)^{-1} (\phi_k)_2^\hyb (t)$ of the hybrid Bergman metric is the complex balanced metric on $\calL_t$ at level $k$.
As noted in \pref{rm: NA balanced is independent of finite base change}, our NA balanced metric is independent of the choice of finite base change.
 Therefore,
we obtain the following:
\begin{Thm}
\label{th: continuity of hybrid balanced metric}
For all sufficiently large $k$ (in fact, for $k\geq 3$ in the present case),
after possibly making a finite base change, 
    there exists a hybrid cpsh metric $\phi_k^\hyb\in \CPSH (\calL^\hyb)$ such that,
    for each $t\in \Delta^*$, the restriction
    $\phi_{k,t}^\hyb:=c(t)^{-1}\phi_k^\hyb(t)$ is the complex balanced metric on $\calL_t$ at level $k$  and, for $t=0$, the restriction $\phi_k^\hyb(0)$ is the NA balanced metric on $L^\na$ at level $k$.
\end{Thm}

In particular,
\pref{th: continuity of hybrid balanced metric}  implies that the non-Archimedean balanced metric $\phi_k$ on $L^\na$ is obtained as the limit of the rescaled complex balanced metrics $c(t)\phi_{k,t}$ on $c(t)\calL_t$
as $t\to 0$. 
This suggests a possible alternative definition of non-Archimedean balanced metrics. Thus, studying the following question might provide an approach to \pref{q: existence of NA balanced}:

\begin{Q.}

Does \pref{th: continuity of hybrid balanced metric} hold for any degenerating family of polarized Calabi--Yau varieties? More generally, does this hold for any degenerating family of asymptotically Chow stable polarized varieties?
\end{Q.}
\begin{Rem}\label{rm: non maximally degenerate case}
This may be regarded as a finite-dimensional analogue of Li’s result for Calabi--Yau metrics \cite[Theorem~1.2]{Li25}.  Note that, unlike in \pref{q: NA estimate for NACY}, when studying
a family
$(\calX,\calL)$, the family carries Archimedean data in the sense of \cite{BJ17}, which allows us to define an appropriate NACY metric on $\calL_K^\na$ even when  $(\calX,\calL)$ is not maximally degenerate. Thus, we may pose  \pref{q: NA estimate for NACY}
in this more general setting.
\end{Rem}

\end{document}